\documentclass[12pt]{amsart}
\usepackage{amssymb}  % defines \nmid, for example
\usepackage{latexsym} % for \Box
\usepackage[all]{xy}
\usepackage{comment,url}
\usepackage[pdftex,colorlinks=true,%
            pdftitle={Explicit n-descent on elliptic curves. III. Algorithms},%
            pdfauthor={J.E. Cremona, T.A. Fisher, C. O'Neil, D. Simon, M. Stoll}]%
           {hyperref}
\newcommand{\Sha}{\mbox{\wncyr Sh}}
\newcommand{\Z}{{\mathbb Z}}
\newcommand{\Q}{{\mathbb Q}}
\def\C{\mathbb C}
\def\O{{\mathcal O}}
\def\bb{{\mathfrak b}}
\def\cc{{\mathfrak c}}
\def\pp{{\mathfrak p}}
\def\ord{\operatorname{ord}}
\def\Norm{\operatorname{Norm}}
\def\Disc{\operatorname{Disc}}
\newcommand{\av}{{\mathbf a}}
\newcommand{\bv}{{\mathbf b}}

\newcommand{\OO}{{\mathcal{O}}}

\newcommand{\R}{{\mathbb R}}
\newcommand{\F}{{\mathbb F}}
\newcommand{\PP}{{\mathbb P}}
\newcommand{\BA}{{\mathbb A}}

\newcommand{\Kbar}{\overline{K}}

\newcommand{\Magma}{{\sf MAGMA}}

\newcommand{\Trd}{\operatorname{Trd}}
\newcommand{\Nrd}{\operatorname{Nrd}}
\newcommand{\GL}{\operatorname{GL}}

\newcommand{\Gal}{\operatorname{Gal}}
\newcommand{\Aut}{\operatorname{Aut}}
\newcommand{\End}{\operatorname{End}}

\newcommand{\Map}{\operatorname{Map}}
\newcommand{\Sym}{\operatorname{Sym}}
\newcommand{\Sel}{\operatorname{Sel}}
\newcommand{\Pic}{\operatorname{Pic}}
\newcommand{\Mat}{\operatorname{Mat}}
\newcommand{\im}{\operatorname{im}}
\newcommand{\res}{\operatorname{res}}
\newcommand{\eps}{\varepsilon}
\newcommand{\Gm}{\mathbb{G}_{\text{\rm m}}}
\newcommand{\isom}{\cong}

\newcommand{\ra}{\longrightarrow}

\newcommand{\Tr}{\operatorname{Tr}}
\newcommand{\tr}{\operatorname{tr}}
\newcommand{\divv}{\operatorname{div}}
\newcommand{\diw}{\operatorname{div}}

\newcommand{\Rbar}{{\overline{R}}}
\newcommand{\Qbar}{{\overline{\Q}}}
\newcommand{\Ob}{\operatorname{Ob}}
\newcommand{\Br}{\operatorname{Br}}

\newcommand{\GQ}{{G_\Q}} % {{\Gal(\overline{\Q}/\Q)}}
\newcommand{\GK}{{G_K}}  % {{\Gal(\overline{K}/K)}}
\newcommand{\GR}{{G_\R}} % {{\Gal(\C/\R)}}
\newcommand{\To}{\longrightarrow}
\newcommand{\Cl}{\operatorname{Cl}}
\newcommand{\SSS}{{Section~}}

\newcommand{\origin}{O} % Origin on E
\newcommand{\SSet}{\mathcal{S}} % a set of primes
\newcommand{\coordinate}{coordinate} % co-ordinate
\newcommand{\coordinates}{coordinates} % co-ordinates
\newcommand{\ie}{i.e.,\ } % \emph{i.e.},
\newcommand{\bidegree}{bidegree} % bi-degree
\newcommand{\bihomogeneous}{bihomogeneous} % bi-homogeneous

\newfont{\wncyr}{wncyr10 at 12pt}
\newfont{\wncyrten}{wncyr10 at 10pt}

\newenvironment{Proof}{\par\noindent{\sc Proof:}}%
                      {\hspace*{\fill}\nobreak$\Box$\par\medskip}

\newcounter{nootje}
\newtheorem{Theorem}{Theorem}[section]
\newtheorem{Proposition}[Theorem]{Proposition}
\newtheorem{Lemma}[Theorem]{Lemma}

\theoremstyle{definition}
\newtheorem{Definition}[Theorem]{Definition}
\newtheorem{Remark}[Theorem]{Remark}

\newtheorem{Example}[Theorem]{Example}

\begin{document}

\date{23rd August, 2012}
\title[Explicit $n$-descent on elliptic curves]%
{Explicit $n$-descent on elliptic curves\\III. Algorithms}

\author{J.E.~Cremona} 
\address{Mathematics Institute,
         University of Warwick,
         Coventry CV4 7AL, UK}
\email{J.E.Cremona@warwick.ac.uk}

\author{T.A.~Fisher}
\address{University of Cambridge,
         DPMMS, Centre for Mathematical Sciences,
         Wilberforce Road, Cambridge CB3 0WB, UK}
\email{T.A.Fisher@dpmms.cam.ac.uk}

\author{C. O'Neil}
\address{New York}
\email{cathy.oneil@gmail.com}

\author{D. Simon}
\address{Universit\'e de Caen, Campus II - Boulevard Mar\'echal Juin,
BP 5186--14032, Caen, France}
\email{Denis.Simon@math.unicaen.fr}

\author{M.~Stoll}
\address{Mathematisches Institut, Universit\"at Bayreuth, 95440
  Bayreuth, Germany}
\email{Michael.Stoll@uni-bayreuth.de}

\renewcommand{\theenumi}{\roman{enumi}}

\begin{abstract} 
  This is the third in a series of papers in which we study the
  $n$-Selmer group of an elliptic curve, with the aim of representing
  its elements as curves of degree $n$ in~$\PP^{n-1}$.  The methods
  we describe are practical in the case $n=3$ for elliptic curves over
  the rationals, and have been implemented in \Magma.

  One important ingredient of our work is an algorithm for 
  trivialising central simple algebras. This is of independent interest:
  for example, it could be used for parametrising Brauer-Severi surfaces.
\end{abstract}

\maketitle

\enlargethispage{0.5ex}

\tableofcontents 

%=========================================================================

\section{Introduction}

Descent on an elliptic curve $E$, defined over a number field $K$,
is a method for obtaining information about
both the Mordell-Weil group $E(K)$
and the Tate-Shafarevich group $\Sha(E/K)$.
Indeed for each integer $n\ge2$ there is an exact sequence
$$ 0 \to E(K)/nE(K) \to \Sel^{(n)}(E/K) \to \Sha(E/K)[n] \to 0 $$
where $\Sel^{(n)}(E/K)$ is the $n$-Selmer group.

This is the third in a series of papers in which
we study the $n$-Selmer group with the aim of representing its
elements (when $n \ge 3$) as curves of degree~$n$ in~$\PP^{n-1}$.
Having this representation allows searching for rational
points on $C$ (which in turn gives points in $E(K)$, since
$C$ may be seen as an $n$-covering of $E$) and is a first step towards
doing higher descents. A further application is to the study of
explicit counter-examples to the Hasse Principle.

The Selmer group $\Sel^{(n)}(E/K)$ is a subgroup of the Galois
cohomology group $H^1(K,E[n])$, which parametrises the $n$-coverings
of $E$.  An $n$-covering $\pi : C \to E$ represents a Selmer group
element if and only if $C$ is everywhere locally soluble, \ie
$C$ has $K_v$-rational points for each completion $K_v$ of~$K$. In this case
it was shown by Cassels~\cite{Cassels-IV-Hauptvermutung} that $C$ admits
a $K$-rational divisor $D$ of degree $n$.
When $n > 2$, we can then use the complete linear system $|D|$ to embed $C$ in
$\PP^{n-1}$. Following the terminology in~\cite[Section 1.3]{paperI},
the image is called a genus one normal curve of degree~$n$.
(For $n=2$ we obtain
instead a double cover $C\to\PP^1$.)  More precisely, Cassels'
argument shows that $\Sel^{(n)}(E/K)$ is contained in the `kernel' of the
obstruction map $\Ob : H^1(K,E[n]) \to \Br(K)$.

In the first paper of this series \cite{paperI} we gave a list of
interpretations of $H^1(K,E[n])$ and of the obstruction map.  Then we
showed how, given $\xi \in H^1(K,E[n])$, to explicitly represent
$\Ob(\xi)$ as a central simple algebra $A$ of dimension $n^2$ over
$K$, by giving structure constants for~$A$; we call $A$ the {\em
  obstruction algebra}. In the case $\xi \in \Sel^{(n)}(E/K)$, we have
$A \isom \Mat_n(K)$. Assuming the existence of a ``Black Box'' to
compute such an isomorphism explicitly, a process we call {\em
  trivialising the obstruction algebra}, we then outlined three
algorithms to compute equations for $C \subset \PP^{n-1}$. These were
called the Hesse pencil, flex algebra, and Segre embedding methods.

In the second paper \cite{paperII} we developed the Segre embedding
method. In this paper we are again concerned with the Segre embedding
method.
We give an outline of the work in the earlier papers, and then give
further details of the algorithms. In particular, taking $K= \Q$ and
$n=3$, we explain our method for trivialising the obstruction algebra
(see \SSS\ref{blackbox}).

Returning to general~$n$, we  observe that
\[ \Sel^{(mn)}(E/K) \isom \Sel^{(m)}(E/K) \times \Sel^{(n)}(E/K) \]
whenever $m$ and~$n$ are coprime. Therefore 
for the purposes of computing Selmer groups we can restrict our
attention to prime powers~$n$. 
Then, if $n = p^f$ with $f \ge 2$, the most
efficient way to proceed seems to be to first recursively compute
$\Sel^{(p^{f-1})}(E/K)$, to realise the elements of that Selmer group
as suitable covering curves (using the methods of this paper, for
example), and then to compute the fibres of the natural map
$\Sel^{(p^f)}(E/K) \to \Sel^{(p^{f-1})}(E/K)$ via $p$-descents on
these covering curves. This has been worked out for $p = 2$ and $f =
2$ by Siksek~\cite{Siksek}, Merriman, Siksek and Smart~\cite{MSS},
Cassels~\cite{Cassels4} and Womack~\cite{Womack}; for $p = 2$ and $f =
3$ by Stamminger~\cite{Stamminger}; and for odd~$p$ and $f = 2$ by
Creutz~\cite{Creutz}.

In Sections~\ref{sec:w2} and~\ref{sec:w1} we recall the construction
of an \'etale algebra (that is, a product of number fields) $R$
and a homomorphism $\partial : R^\times \to (R \otimes R)^\times$
such that
$\Sel^{(n)}(E/K)$ may be realised either as a subgroup of $(R \otimes
R)^\times/ \partial R^\times$, with elements represented by $\rho \in
(R \otimes R)^\times$, or as a subgroup of $R^\times/(R^\times)^n$,
with elements represented by $\alpha \in R^\times$. The first
of these works for any $n \ge 2$ and is better suited to the
computation of the obstruction algebra and equations for $C$. The
second only works for prime~$n$, but is better suited to computing the
Selmer group itself, in that the class group and unit calculations are
more manageable.  So in \SSS\ref{sec:convert}, we discuss how to
convert from one representation to the other (from~$\alpha$
to~$\rho$). The Segre embedding method is then reviewed in
Sections~\ref{compcov} and~\ref{improve_eqns}.

It is sometimes convenient to assume $n >2$. For example, when $n=2$
the map $C \to \PP^1$ is a double cover rather than an embedding.
However, if the Segre embedding method is suitably interpreted in the
case $n=2$, then it corresponds exactly to the classical number field
method\footnote{ An alternative method for $2$-descent over 
  $K = \Q$, based on invariant theory, 
  is implemented in \texttt{mwrank} \cite{mwrank}, and is
  competitive over a large range of curves, but seems to become
  impractical in other cases: when $K$ is a larger number field, or
  when $n > 2$.}  for $2$-descent. This is explained in
\SSS\ref{twodesc}.  In \SSS\ref{threedesc}, we give an equally
explicit description of our algorithms in the case $n=3$, assuming the
action of Galois on $E[3]$ is generic.
We do not give full details of the modifications required to handle 
the other Galois actions  as this would be unduly tedious,
though each case had to be handled in detail in our \Magma\
implementation.

Starting with $\alpha \in R^\times$, we write down structure
constants for the obstruction algebra~$A$. Then we trivialise the
algebra~$A$.
Using the trivialisation, we obtain a plane cubic $C \subset \PP^2$.
Now, the element $\alpha$ is typically of very large height: it comes
out of a class group and unit calculation that involves many random
choices.  Consequently, the first equation for~$C$ which we obtain is
a ternary cubic with enormous coefficients.  In order to obtain a more
reasonable equation, we finally use our algorithms for minimisation
and reduction (see \cite{minred234}) to make a good change of
\coordinates.

After we wrote our initial implementation for $K=\Q$ and $n=3$, 
it became clear that the algorithm could be improved by carrying out steps
equivalent to minimisation and reduction at an earlier stage.
Firstly, $\alpha$ should be replaced by a good representative
modulo $n$th powers, as has already been described in \cite[Section
  2]{FisherANTS2008}. Then we should choose a good basis for the
obstruction algebra, so as to make the structure constants small
integers. This is described in \SSS\ref{compobs}, and makes
trivialising the obstruction algebra very much easier. (In fact, this
trivialisation is again a problem of ``minimisation and reduction''
type.) As a result, the algorithms in \cite{minred234}, although still
required, do not need to work so hard.

In \SSS\ref{blackbox} we describe our methods for trivialising the
obstruction algebra. Since our methods are of independent interest, we
have made this section self-contained. For instance our methods could
be used to improve the algorithm in \cite{GHPS} for parametrising
Brauer-Severi surfaces.

One peculiar feature of the Segre embedding method is that in our
initial implementation (for $K= \Q$ and $n=3$) it was necessary to
multiply by a ``fudge factor'' $1/y$ to ensure that the projection of
$C \subset \PP(A)$ to the trace~$0$ subspace is contained in the rank
$1$ locus. The need for this factor was justified by a generic
calculation specific to the case $n=3$. In \SSS\ref{exteqns} we
give a better explanation, based on the 
theory in \cite{paperII}, that works for all odd integers~$n$.

Finally, we give examples of the algorithm in practice, in
\SSS\ref{sec:examples}.  One application of our work is that it
can help find generators of large height on an elliptic curve.
Indeed the logarithmic height of a rational point on an $n$-covering 
is expected to be smaller by a factor $2n$ compared to its 
image on the elliptic curve. (See \cite{FisherSills} for a precise statement.)
Our work in the case $n=3$ is a starting point both for the
work on $6$-{} and $12$-descent in \cite{6and12}, and for the work on
$9$-descent in \cite{Creutz}.  In \SSS\ref{sec:examples} we
instead use our methods to exhibit some explicit elements of
$\Sha(E/\Q)[3]$, and to give some  examples 
illustrating that the kernel of the obstruction map on 3-coverings is 
not a group. 
The use of our methods to compute some explicit elements of 
$\Sha(E/\Q)[5]$ is described in \cite{desc5}. For this we use that
the Hesse pencil method (described in \cite[Section 5.1]{paperI} 
in the case $n=3$) generalises to the case $n=5$ as described 
in \cite[Section 12]{hessians}.

Our algorithms have been implemented in (and contributed to)
\Magma\ \cite[Version 2.13 and later]{Magma} for $K = \Q$ and $n=3$.  
A first version of the implementation was written by
Michael Stoll; it was restricted to the case of a transitive
Galois action on the points of order~$3$. 
Steve Donnelly extended the part that computes the
$3$-Selmer group as an abstract group to cover all 
possible Galois actions,
and Tom Fisher re-worked and extended the part that turns
abstract Selmer group elements into plane cubic curves, so that it
also works for all possible Galois actions.

These programs are currently specific to the case $K= \Q$. The two main
obstacles to extending them to general number fields are as follows.
Firstly, it would be necessary (unless the Galois action on $E[3]$ is
smaller than usual) to compute class group and units over number
fields of larger degree.  Secondly, we do not have a suitable theory
of lattice reduction over number fields. Notice that our algorithms
over $K=\Q$ are dependent on the LLL-algorithm, first in \cite[Section
2]{FisherANTS2008}, then in Sections~\ref{compobs} and~\ref{blackbox} 
and finally in \cite{minred234}.  It is
possible that the algorithms described in \cite{Fieker-Stehle} could be
used here, but we have not yet investigated this further.

%=========================================================================

\section{Algorithms in outline}
\label{sec:overview}

In this section we give an outline of the algorithms used in our
implementation of explicit 3-descent on elliptic curves over~$\Q$.
However, as far as possible in this overview, we keep to the case
where $n$ is general and the base field is a general number field~$K$.
Details specific to the case $n = 3$ can be found in
\SSS\ref{threedesc} below.  Our description here is based on what
was called the `Segre embedding method' in~\cite{paperI,paperII}.  We
begin with a review of the computation of the Selmer group as an
abstract group. When we talk about the `generic case' below, we refer
to the situation when the action of the Galois group~$G_K$ on~$E[n]$
induces a surjective homomorphism of~$\GK$ onto $\Aut_{\Z/n\Z}(E[n])
\isom \GL_2(\Z/n\Z)$. When $K = \Q$ and $E$ is fixed and non-CM, this
will be the case for all but finitely many prime values of~$n$. 
It will also be true for the generic elliptic curve 
$y^2 = x^3 + a x + b$ over~$\Q(a,b)$.

%----------------------------------------------------------------------

\subsection{Computation of the Selmer group I: The \'etale algebra}
\label{compsel}

Let $E$ be an elliptic curve over a field $K$. (We will take $K$ to be
a number field later.) We fix a Weierstra{\ss} equation for~$E$ and
denote the \coordinate{} functions with respect to this equation by $x$
and~$y$.  Let $n \ge 2$ be an integer not divisible by the
characteristic of~$K$. Writing $\Map_K$ for the space of 
Galois equivariant maps, we let $R = \Map_K(E[n],\Kbar)$ be the \'etale
algebra of~$E[n]$. Then $E[n] = {\operatorname{Spec}} \, R$,
and the algebra $R$ splits as a product of 
finite extensions of~$K$ corresponding to the Galois orbits on~$E[n]$, where
the component corresponding to the orbit of $T \in E[n]$ is~$K(T)$,
the field of definition of~$T$. There is always a splitting $R = K
\times L$ with $K$ corresponding to the singleton orbit~$\{\origin\}$
and $L = \Map_K(E[n] \setminus \{\origin\}, \Kbar)$.  If $n$ is a
prime~$p$, then generically the Galois action is transitive on the
points of order~$p$, and $L/K$ is a field extension of degree~$p^2-1$.

The tensor product $R \otimes_K R$ is the \'etale algebra of~$E[n]
\times E[n]$. We denote by $\Sym_K^2(R)$ the subalgebra consisting of
symmetric functions:
\[ \Sym_K^2(R) = \{\rho \in R \otimes_K R \mid
                    \rho(T_1,T_2) = \rho(T_2,T_1)
                     \text{\; for all $T_1, T_2 \in E[n]$}\}
\]
This is the \'etale algebra of the set of unordered pairs of
$n$-torsion points.  As before, these algebras split into products of
finite field extensions of~$K$ corresponding to the Galois orbits on
$E[n] \times E[n]$ and on the set of unordered pairs of $n$-torsion
points, respectively. The algebra $\Sym_K^2(R)$ contains a factor
corresponding to unordered bases of~$E[n]$ as a $\Z/n\Z$-module.  When
$n$ is a prime~$p$, then generically, the Galois group acts
transitively on these bases, and the corresponding factor
of~$\Sym^2_K(R)$ is a field extension of~$K$ of degree
$(p^2-1)(p^2-p)/2$.

The group law 
$E[n] \times E[n] \to E[n]$ corresponds to the comultiplication
$\Delta_K : R \to \Sym_K^2(R) \subset R \otimes_K R$.
We write $\Tr_K: R \otimes_K R \to R$ for
the trace map obtained by viewing $R \otimes_K R$ as an $R$-algebra
via~$\Delta_K$. In terms of maps we have
\[ \Delta_K(\alpha): (T_1,T_2) \mapsto \alpha(T_1+T_2)
    \text{\quad and\quad}
    \Tr_K(\rho) : T \mapsto \sum_{T_1+T_2 = T} \rho(T_1,T_2) \,.
\]

%----------------------------------------------------------------------

\subsection{Computation of the Selmer group II: Using $w_2$}
\label{sec:w2}

We define
\[
\begin{aligned}
\partial_K : R^\times &\To \Sym_{K}^2(R)^\times \, \\
\text{by}\qquad
\alpha &\longmapsto \frac{\alpha \otimes \alpha}{\Delta_K(\alpha)} 
               = \left( (T_1,T_2)
             \mapsto \frac{\alpha(T_1) \alpha(T_2)}%
                          {\alpha(T_1+T_2)} \right).
\end{aligned}
\]
In~\cite[p.~138]{paperI}, we defined another map $\partial$, which we here
denote $\partial^{(2)}_K$ to avoid confusion. It is given by
\begin{align*}
  \partial^{(2)}_K : (R \otimes_K R)^\times &\To (R \otimes_K R \otimes_K R)^\times \\
    \rho &\longmapsto \Bigl((T_1,T_2,T_3) \mapsto
                     \frac{\rho(T_1,T_2) \rho(T_1+T_2,T_3)}%
                          {\rho(T_1,T_2+T_3) \rho(T_2,T_3)}\Bigr) \,.
\end{align*}
We let $H_K = \Sym^2_K(R)^\times \cap \ker \partial^{(2)}_K$.

Let $\Rbar = R \otimes_K \Kbar$ (which is the \'etale algebra of~$E[n]$
over~$\Kbar$), and let $\Sym_{\Kbar}^2(\Rbar)$ be the \'etale algebra
over~$\Kbar$ of the set of unordered pairs of $n$-torsion points.
Similarly, we write $\overline{H}$ for~$H_{\Kbar}$.

Let $w : E(\Kbar)[n] \to \Rbar^\times$ be given by
\[ w(S) : T \longmapsto e_n(S, T) \] where $e_n : E[n] \times E[n] \to
\mu_n$ denotes the Weil pairing.  Then it is easily seen that the
image of~$w$ equals the kernel of~$\partial_{\Kbar}$.  We showed
in~\cite{paperI} that the following is an exact sequence of 
$\GK$-modules:
\[ 0 \To E(\Kbar)[n] \stackrel{w}{\To} \Rbar^\times
      \stackrel{\partial_{\Kbar}}{\To} \overline{H} \To 0 \,.
\]
Taking cohomology, this gives an isomorphism
\[ w_2 : H^1(K, E[n]) \To H_K/\partial R^\times \,, \]
see~\cite[Lemmas~3.2 and~3.5]{paperI}.  For the construction of
explicit $n$-coverings representing elements $\xi \in H^1(K, E[n])$
(which, in our intended application, will be elements of the
$n$-Selmer group), we will need an element $\rho \in H_K$ whose image
in $H_K/\partial R^\times$ is the image under~$w_2$ of~$\xi$.

In principle, we could use $w_2$ to compute such a set of
representatives of the elements of~$\Sel^{(n)}(E/K)$ directly, as we
now describe. We now assume that $K$ is a number field. We abbreviate
$H_K$ to~$H$ (note that what is called~$H$ in~\cite{paperI} would be
$H/\partial R^\times$ in the notation used here) and usually drop the
subscripts on $\Delta$, $\partial$, etc.

Recall the Kummer exact sequence
\[ 0 \To E(K)/n E(K) \stackrel{\delta}{\To} H^1(K, E[n]) \To H^1(K,
E)[n] \To 0 \,. \] For a place~$v$ of~$K$, we write $R_v = R \otimes_K
K_v$ and $H_v = H_{K_v}$.  We denote the canonical maps $H \to H_v$
and $H/\partial R^\times \to H_v/\partial R_v^\times$ by~$\res_v$. 
The maps corresponding to $\delta$ and~$w_2$ that we obtain 
by working over~$K_v$ are denoted by $\delta_v$ and~$w_{2,v}$.

By the definition of the $n$-Selmer group and the fact that $w_2$ is an isomorphism, we have
\begin{align*}
  w_2&\bigl(\Sel^{(n)}(E/K)\bigr) \\
    &= \{\rho \in H/\partial R^\times \mid
          \res_v(\rho) \in \im(w_{2,v} \circ \delta_v)
          \text{\; for all places $v$ of~$K$}\} \,.
\end{align*}
According to~\cite{SchaeferStoll}, the image of~$\delta_v$ is the
unramified subgroup of $H^1(K_v, E[n])$ unless $v$ is infinite and $n$ is even,
or $v$ divides~$n$, or the Tamagawa number of~$E$ at~$v$ is not coprime to~$n$.
Let $\SSet$ be the set of places of~$K$ that fall into one of these categories.
If $v$ is a place of~$K$ and $\rho \in H$, we say that $\rho$ is \emph{unramified
at~$v$} if $\rho\, \partial R^\times = w_2(\xi)$ with $\xi \in H^1(K, E[n])$
unramified at~$v$ (\ie
\[ \res_v(\xi) \in \ker\bigl(H^1(K_v, E[n]) \to H^1(K_v^{\text{nr}}, E[n])\bigr) \]
where $K_v^{\text{nr}}$ is the maximal unramified extension of~$K_v$).
This is equivalent to saying that
the extension $R(\gamma)/R$ of \'etale algebras is unramified
at~$v$ for some $\gamma \in \Rbar^\times$ satisfying $\partial \gamma = \rho$.
Writing $H_\SSet$ for the subgroup of elements unramified outside~$\SSet$ and
$\widetilde{H}_\SSet$ for the image of~$H_\SSet$ in~$H/\partial R^\times$, we then have
\begin{align*}
  w_2\bigl(\Sel^{(n)}(E/K)\bigr)
    &= \{\rho \in \widetilde{H}_\SSet \mid
          \res_v(\rho) \in \im(w_{2,v} \circ \delta_v)
          \text{\; for all $v \in \SSet$}\} \,.
\end{align*}
For an \'etale $K$-algebra~$A$, write $U_\SSet(A)$ for the group of $\SSet$-units
of~$A$, $I_\SSet(A)$ for the group of ideals of~$A$ supported outside~$\SSet$
and $\Cl_\SSet(A)$ for the $\SSet$-class group of~$A$. Then there is an exact sequence
\[ 0 \To U_\SSet(A) \To A^\times \To I_\SSet(A) \To \Cl_\SSet(A) \To 0 \,. \]
The map~$\partial$ induces a homomorphism from this exact sequence
for $A = R$ to the corresponding sequence for $A = \Sym^2_K(R)$.
Applying the Snake Lemma to the commutative diagram (with exact rows)
\[ \xymatrix{  & R^\times \ar[r]^{\partial} \ar[d]
               & H \ar[r] \ar[d]
               & H /\partial R^\times \ar[r] \ar[d]
               & 0 \\
              0 \ar[r]
               & I_\SSet(R) \ar[r]^-{\partial}
               & I_\SSet\bigl(\Sym^2_K(R)\bigr) \ar[r]
               & \dfrac{I_\SSet\bigl(\Sym^2_K(R)\bigr)}{\partial I_\SSet(R)} \ar[r]
               & 0
            }
\]
and observing that $\widetilde{H}_\SSet$ is the kernel of the right-most vertical
map (this uses that $\SSet$ contains all primes dividing~$n$), we obtain
an exact sequence
\[ 0 \To \frac{U_\SSet\bigl(\Sym^2_K(R)\bigr) \cap H}{\partial U_\SSet(R)}
     \To \widetilde{H}_\SSet
     \To \Cl_\SSet(R)^0 \To 0 
\]
where $\Cl_\SSet(R)^0 = \ker\bigl(\partial : \Cl_\SSet(R) \to
\Cl_\SSet(\Sym^2_K(R))\bigr)$.  There are algorithms for computing
$\SSet$-unit groups and $\SSet$-class groups of number fields (see for
example~\cite[7.4.2]{CohenGTM193}), which can be applied to the
constituent fields of $R$ and~$\Sym^2_K(R)$. Based on these and the
exact sequence above, we can compute an explicit set of generators
of~$\widetilde{H}_\SSet$.

In order to turn the description of~$\Sel^{(n)}(E/K)$ above
into an algorithm, we need to be able
to evaluate $w_{2,v} \circ \delta_v$. This can be done as follows.

Let $T \in E[n]$. Then there is a rational function $G_T \in
K(T)(E)^\times$ such that
\[ \diw(G_T) = [n]^*(T) - [n]^*(\origin)
             = \sum_{P :\  nP = T} (P) - \sum_{Q :\  nQ = \origin} (Q) \,.
\]
These functions have the property that $G_T(P+S) = e_n(S,T) G_T(P)$
for all $S \in E[n]$ and $P \in E$, provided both sides are defined.
We can choose them in such a way that the map $G : T \mapsto G_T$ is
Galois-equivariant.  Then we can interpret $G$ as an element of
$R(E)^\times$.  Here, $R(E) = K(E) \otimes_K R$; its elements are
$\GK$-equivariant maps from $E[n]$ into $\overline{K}(E)$.  Then we
have $G(P + T) = w(T) G(P)$ for $P \in E \setminus E[n^2]$ and $T \in
E[n]$.

For $T_1, T_2 \in E[n]$ define rational functions~$r_{T_1,T_2}$
(compare~\cite[p.~67]{paperII}) by
\[ r_{T_1,T_2} = \left\{\begin{array}{cl}
                   1 & \text{\quad if $T_1 = \origin$ or $T_2 = \origin$;} \\
                   x - x(T_1) & \text{\quad if $T_1 + T_2 = \origin$, $T_1 \neq \origin$;}\\
                   \dfrac{y + y(T_1+T_2)}{x - x(T_1+T_2)} - \lambda(T_1,T_2)
                     & \text{\quad otherwise,}
                 \end{array}\right.
\]
where $\lambda(T_1,T_2)$ denotes the slope of the line joining $T_1$
and~$T_2$ (or the slope of the tangent line at $T_1 = T_2$ if the
points coincide).  Just as before, we can package these functions into
a single element $r \in \Sym^2_K(R)(E)^\times$.

Now consider the following diagram with exact rows.
\begin{equation} \label{diagram-w2}
   \xymatrix{ 0 \ar[r] & E(\Kbar)[n] \ar[r] \ar@{=}[d]
                       & E(\Kbar) \ar[r]^{n} \ar@{-->}[d]^{G}
                       & E(\Kbar) \ar[r] \ar@{-->}[d]^{r} & 0 \\
              0 \ar[r] & E(\Kbar)[n] \ar[r]^-{w} 
                       & \overline{R}^\times \ar[r]^-{\partial}
                       & \overline{H} \ar[r] & 0
            }
\end{equation}
The dashed arrows indicate the partially defined maps given by
evaluating the tuples of rational functions $G$ and $r$.  The
right-hand square commutes by~\cite[Proposition~3.2]{paperII},
provided that we scale the~$G_T$ as in the proof of Proposition~3.1
of~\cite{paperII}.  The left-hand square commutes in the sense that
the actions of~$E[n]$ on~$E(\Kbar)$ by translation and
on~$\Rbar^\times$ by multiplication via~$w$ are related by~$G$, \ie
$G(P + T) = w(T) G(P)$.  Chasing through the definitions of the
connecting homomorphisms $\delta$ and~$w_2$, we now easily find the
following.

\begin{Proposition} \label{Propw2}
  The composition $w_2 \circ \delta$ is induced
  by the map $r : E(K) \setminus E[n] \To H$.
\end{Proposition}

We can extend $r$ to divisors on~$E$ with support disjoint from~$E[n]$
by defining
\[ r\Bigl(\sum_P n_P(P)\Bigr) = \prod_P r(P)^{n_P} \,. \]
Then for a principal divisor $D = \diw(h)$ we find by Weil reciprocity
\begin{align*}
  r_{T_1,T_2}(D) &= r_{T_1,T_1}\bigl(\diw(h)\bigr)
                  = h\bigl(\diw(r_{T_1,T_2})\bigr) \\
                 &= \frac{h(T_1) h(T_2)}{h(\origin) h(T_1+T_2)}
                  = \frac{1}{h(\origin)} (\partial h|_{E[n]})(T_1, T_2) \,.
\end{align*}
We can scale $h$ so that $h(\origin) = 1$; then
\[ r\bigl(\diw(h)\bigr) = \partial h|_{E[n]} \in \partial R^\times \]
if $h \in K(E)^\times$. Therefore we obtain a well-defined map
\[ \widetilde{r} : E(K) \isom \Pic^0(E/K) \To H/\partial R^\times \,, \]
and we have
\[ w_2 \circ \delta = \widetilde{r} \,. \] This construction is valid
over any field~$K$ of characteristic not dividing~$n$; in particular,
it can be applied over~$K_v$ to find $\im(w_{2,v} \circ \delta_v)$.
In~\cite{SchaeferStoll} there is a discussion of how to compute images
under local descent maps, which applies \emph{mutatis mutandis} to the
situation at hand.

\bigskip 

The following theorem summarises this section.

\begin{Theorem}
  Let $K$ be a number field, $E/K$ an elliptic curve, and $n \ge 2$.
  There is an efficient algorithm that computes $\Sel^{(n)}(E/K)$, given 
  knowledge of class and unit groups of the number fields
  $K(\{T_1, T_2\})$, where $\{T_1, T_2\}$ runs through unordered pairs of
  $n$-torsion points of~$E$.
\end{Theorem}

\begin{proof}
  The algorithm proceeds in the following steps.
  \begin{enumerate}\renewcommand{\theenumi}{\arabic{enumi}.}
    \item Let $\SSet$ be the set of places~$v$ of~$K$ that divide~$n$
          or such that the Tamagawa number of $E/K_v$ is not coprime
          to~$n$, together with the real places of~$K$ when $n$ is even.
    \item Construct the \'etale algebras $R$ and~$\Sym^2_K(R)$.
    \item Compute an explicit representation of~$\widetilde{H}_\SSet$ as defined
          above.
    \item For each $v \in \SSet$ construct
          $\widetilde{H}_v = H_{K_v}/\partial R_v^\times$,
          together with the map $\res_v : \widetilde{H}_\SSet \to \widetilde{H}_v$, and
          find the local image $\widetilde{r}(E(K_v)) \subset \widetilde{H}_v$.
    \item Compute $\Sel^{(n)}(E/K)$ as
          \[ \bigcap_{v \in \SSet} \res_v^{-1}\bigl(\widetilde{r}(E(K_v))\bigr)
               \subset \widetilde{H}_\SSet \,.
          \]
  \end{enumerate}
  The third step (computation of~$\widetilde{H}_\SSet$) relies on the computation
  of the $\SSet$-class groups and $\SSet$-unit groups of the constituent fields
  of~$\Sym^2_K(R)$, which are of the form $K(\{T_1,T_2\})$ as in the
  statement of the theorem. The $\SSet$-class and $\SSet$-unit groups can be
  computed from the class and unit groups.
\end{proof}

As it stands, this result is rather theoretical, since the number
fields that occur are (in most cases) too large for practical
computations: as mentioned earlier, when $n$ is a prime~$p$, then
usually there is a component of~$\Sym^2_K(R)$ that is a field
extension of degree $(p^2-1)(p^2-p)/2$ of~$K$. Even when $n = 3$ (and the
degree is~$24$), this makes unconditional computation of the class group nearly
impossible in reasonable time with currently available implementations.
However, when $n$ is a prime, there is a
better alternative, which we describe next.

%----------------------------------------------------------------------

\subsection{Computation of the Selmer group III: Using $w_1$}
\label{sec:w1}

There is a group homomorphism (recalled from~\cite{paperI})
\[ w_1 : H^1(K,E[n]) \to R^\times/(R^\times)^n \,, \]
which is obtained by applying cohomology to the exact sequence
\begin{equation}
\label{eqn-wd}
 0 \To E[n] \stackrel{w}{\To} \mu_n(\Rbar) 
   \stackrel{\partial}{\To} \partial( \mu_n(\Rbar) ) \To 0 \,.
\end{equation}
It is shown in \cite{DSS} and~\cite{SchaeferStoll} that when $n$ is a
prime~$p$, the map $w_1$ is injective (for any field~$K$ of
characteristic different from~$p$), and a description of the image is
given. This can be used for computing the $n$-Selmer group as a
subgroup of $R^\times/(R^\times)^n$. 

Let $v$ be a place of~$K$.  Then there is an analogous homomorphism
\[ w_{1,v} : H^1(K_v, E[n]) \to R_v^\times/(R_v^\times)^n \,. \]
The maps fit together in a commutative diagram (whose rows are usually not exact)
\[ \xymatrix{ E(K) \ar[r]^-{\delta} \ar[d]
                & H^1(K, E[n]) \ar[d] \ar[r]^-{w_1}
                & R^\times/(R^\times)^n \ar[d]^{\res_v} \\
              E(K_v) \ar[r]^-{\delta_v}
                & H^1(K_v, E[n]) \ar[r]^-{w_{1,v}}
                & R_v^\times/(R_v^\times)^n.
            }
\]
If $w_1$ and all the $w_{1,v}$ are injective 
(as happens in the case $n$ is prime), then this tells us that the
$n$-Selmer group can be realised as an abstract group via
\[ \Sel^{(n)}(E/K)
      \isom R(\SSet,n) \cap \im(w_1) \cap
           \bigcap_{v \in \SSet} \res_v^{-1}\bigl(\im(w_{1,v} \circ \delta_v)\bigr) \,;
\]
see~\cite{SchaeferStoll}. Here, $R(\SSet,n) \subset R^\times/(R^\times)^n$ is
the subgroup of elements~$\alpha$ unramified outside~$\SSet$ (\ie such that the
extension $R(\sqrt[n]{\alpha})/R$ of \'etale algebras is unramified outside~$\SSet$).
As before, the group $R(\SSet,n)$ can be determined from a knowledge of the $\SSet$-class
and $\SSet$-unit groups of the various number fields that occur in the splitting
of~$R$ according to the Galois orbits on~$E[n]$.

Note that in~\cite{SchaeferStoll}, the \'etale algebra~$L$ (denoted there
by~$A$) is used instead of $R = K \times L$. 
If $\alpha \in R^\times$ represents some $\xi \in H^1(K,E[n])$,
then $\alpha(\origin) \in (K^\times)^n$, so without loss of generality we
can assume that $\alpha(\origin) = 1$. Therefore all the relevant
information is already contained in the $L$-component.

If $n$ is not a prime, then the map $w_1$ need not be injective: 
see \SSS\ref{sec:convert} for an example with~$n=4$.
For the realisation of~$w_1 \circ \delta$, consider the following diagram
analogous to~\eqref{diagram-w2}:
\[ \xymatrix{ 0 \ar[r] & E(\Kbar)[n] \ar[r] \ar[d]^{w}
                       & E(\Kbar) \ar[r]^{n} \ar@{-->}[d]^{G}
                       & E(\Kbar) \ar[r] \ar@{-->}[d]^{F}
                       & 0 \\
              0 \ar[r] & \mu_n(\overline{R}) \ar[r]
                       & \overline{R}^\times \ar[r]^{n}
                       & \overline{R}^\times \ar[r]
                       & 0
            }
\]
Here, $F \in R(E)^\times$ is the function such that $F_T(nQ) = G_T(Q)^n$ 
for all $T \in E[n]$.
The divisor of~$F_T$ is $n (T) - n (\origin)$, and if
$F(nQ) = G(Q)^n$ with $G \in R(E)^\times$, then $F$ induces a well-defined
map $F : E(K) \setminus E[n] \To R^\times/(R^\times)^n$, independent of
the particular choice of~$F$. In the same way as discussed after
Proposition~\ref{Propw2}, we can extend this map~$F$ to a map on divisors
with support disjoint from~$E[n]$, which only depends on the linear equivalence
class, and therefore gives rise to a homomorphism
\[ \widetilde{F} : E(K) \isom \Pic^0(E/K) \To R^\times/(R^\times)^n \,. \]
This leads to a new proof of the following well-known fact.

\begin{Proposition} \label{Propw1}
  The composition $w_1 \circ \delta : E(K) \to R^\times/(R^\times)^n$
  is given by~$\widetilde{F}$.
\end{Proposition}

Again, this works for any field~$K$ of characteristic not
dividing~$n$, and so we can use it for evaluating the local
maps~$\delta_v$. The algorithm for computing $\Sel^{(n)}(K, E)$ then
proceeds as before, but now working within~$R$ instead
of~$\Sym_K^2(R)$.  The functions~$F_T$ can be evaluated at a given
point using Miller's algorithm~\cite{Miller}, which follows the
computation of $n T$ and keeps track of the functions $r_{T_1,T_2}$
witnessing the intermediate sums; it is not necessary to compute an
expression for~$F_T$ in terms of the \coordinates{} on a Weierstra{\ss}
equations of~$E$ (which will be quite complicated when $n$ is large).
In practice, however, even moderately large~$n$ quickly make
computations infeasible, so $n$ will be rather small, and it is no
problem to work with an explicit expression for~$F_T$ as a function.
Such an expression is also helpful for computing $\eps$ as defined
by~\eqref{defeps} below.

\begin{Theorem} \label{Thmw1} Let $K$
  be a number field, $E/K$ an elliptic curve and $p$ a prime number.
  There is an efficient algorithm that computes
  $\Sel^{(p)}(E/K)$, given knowledge of class and unit groups of the
  number fields $K(T)$, where $T$ runs through points of order~$p$
  on~$E$.
\end{Theorem}

\begin{proof}
  The algorithm proceeds in the following steps (see~\cite{SchaeferStoll}).
  \begin{enumerate}\renewcommand{\theenumi}{\arabic{enumi}.}
    \item Let $\SSet$ be the set of places~$v$ of~$K$ that divide~$p$
          or such that the Tamagawa number of $E/K_v$ is divisible
          by~$p$, together with the real places of~$K$ when $p = 2$.
    \item Construct the \'etale algebra $R$.
    \item Compute an explicit representation 
          of~$R_1 = R(\SSet,p) \cap \im(w_1)$.
    \item For each $v \in \SSet$ construct $H_v = R_v^\times/(R_v^\times)^p$,
          together with the map $\res_v : R(\SSet,p) \to H_v$, and
          find the local image $\widetilde{F}(E(K_v)) \subset H_v$.
    \item Compute $\Sel^{(p)}(E/K)$ as
          \[ R_1 \cap \bigcap_{v \in \SSet}
                        \res_v^{-1}\bigl(\widetilde{F}(E(K_v))\bigr)
               \subset R(\SSet,p) \,.
          \]
  \end{enumerate}
\end{proof}

%----------------------------------------------------------------------

\subsection{Changing algebras}
\label{sec:convert}

For the purpose of constructing explicit models of $n$-coverings representing
the various elements of the $n$-Selmer group, we need to represent the
Selmer group elements by elements $\rho \in H$. So after computing
$\Sel^{(p)}(E/K)$ as in Theorem~\ref{Thmw1}, we need to convert the elements
$\alpha \in R^\times$ we obtain as representatives into elements $\rho \in H$.
Note that for this purpose it is helpful to choose a small 
representative for the class of $\alpha$ up to $n$th powers, by applying
the method in~\cite[\SSS2]{FisherANTS2008} over each constituent field
of $R$.

Recall that $H$ is the subgroup of elements $\rho \in \Sym^2_K(R)^\times$ 
satisfying
\begin{equation} \label{assoc0}
  \rho(T_1,T_2+T_3) \rho(T_2,T_3) = \rho(T_1,T_2) \rho(T_1+T_2, T_3)
  \text{\; for all $T_1,T_2,T_3 \in E[n]$.}
\end{equation}

\begin{Lemma}
\label{lem:convert}
  Let $\alpha \in R^\times$ represent an element in the image
  of $w_1$, \ie $\alpha (R^\times)^n = w_1(\xi)$ for some 
  $\xi \in H^1(K, E[n])$.
  Then there exists $\rho \in H$ satisfying $\partial \alpha = \rho^n$ and 
  \begin{equation} \label{prod}
    \alpha(T) = \textstyle\prod_{i=0}^{n-1} \rho(T,iT)
                  \qquad \text{ for all } T \in E[n].
  \end{equation}
  Moreover if $\rho \in H$ satisfies~(\ref{prod}) then 
  $\partial \alpha = \rho^n$ and 
  $\rho \, \partial R^\times = w_2(\xi')$ 
  for some $\xi' \in H^1(K, E[n])$ with $w_1(\xi) = w_1(\xi')$.
\end{Lemma}
\begin{Proof}
  This is \cite[Lemma 3.8]{paperI}.
\end{Proof}

To convert $\alpha$ to~$\rho$ we first extract an $n$th root of
$\partial \alpha$ in~$\Sym_K^2(R)$.
We then multiply by an $n$th root of unity in~$\Sym_K^2(R)$ to 
find $\rho$ satisfying~\eqref{assoc0} and~\eqref{prod}.
The simplest case, which occurs frequently in practice,
is when $\Sym_K^2(R)$ contains no non-trivial 
$n$th roots of unity. There is then a unique choice of~$\rho$.
In general we can avoid checking all the conditions in~(\ref{assoc0})
by determining in advance the number of solutions for $\rho$.

\begin{Definition}
\label{gammadef}
Let $\Gamma$ be the group (under pointwise operations)
of all maps $\gamma : E[n] \to \mu_n$ satisfying
\[ \frac{\gamma(\sigma T_1) \gamma(\sigma T_2)}{ \gamma 
(\sigma (T_1+ T_2))} =  \sigma \left(\frac{\gamma( T_1)  
\gamma(T_2)}{ \gamma (T_1+T_2)} \right)  \]
for all $\sigma \in \GK$ and $T_1,T_2 \in E[n]$.
\end{Definition}

Let $G \isom \Gal(K(E[n])/K)$ be the subgroup of $\GL_2(\Z/n\Z)$
describing the action of $\GK$ on $E[n]$.
The action of $\GK$ on $\mu_n$ is given by the
determinants of these matrices. Hence $\Gamma$ depends only on $G$.
Indeed, given generators for $G$, it is easy to compute $\Gamma$
using linear algebra over $\Z/n\Z$.
The following lemma shows that the number of solutions for $\rho$
in Lemma~\ref{lem:convert} is $\# \partial \Gamma = (\# \Gamma) / n^2$.

\begin{Lemma}
\label{partialgamma}
\begin{enumerate}
\item There is an exact sequence of abelian groups
\[  0 \ra E[n] \stackrel{w}{\ra}
 \Gamma \stackrel{\partial}{\ra} \partial \Gamma \ra 0. 
\]
\item $\partial \Gamma = \{ \rho \in H : 
  \textstyle\prod_{i=0}^{n-1} \rho(T,iT) = 1 \text{ for all }
T \in E[n] \}$.
\end{enumerate}
\end{Lemma}
\begin{Proof} 
(i) Since $\Gamma \subset \mu_n(\Rbar)$, this is obtained by 
restricting the exact sequence~(\ref{eqn-wd}). We note that if
$\gamma \in w(E[n])$ then $\gamma : E[n] \to \mu_n$ is a group 
homomorphism and so clearly $\gamma \in \Gamma$. \\
(ii) By \cite[Corollary 3.6]{paperI} every $\rho \in H$ 
can be written as $\rho = \partial \gamma$ for some $\gamma \in \Rbar^\times$.
We note that if $\rho = \partial \gamma$ then 
\[\prod_{i=0}^{n-1} \rho(T,iT) = \gamma(T)^n.\] Hence the group on the
right of (ii) consists of elements $\partial \gamma$ where $\gamma \in
\mu_n(\Rbar) = \Map(E[n],\mu_n)$ and $\partial \gamma : E[n] \times
E[n] \to \mu_n$ is Galois equivariant.  From
Definition~\ref{gammadef}, we recognise this group as $\partial
\Gamma$.
\end{Proof}

\begin{Lemma}
\label{kerw1}
The kernel of $w_1 : H^1(K,E[n]) \to R^\times/(R^\times)^n$ 
is isomorphic to 
$(\partial \mu_n(\Rbar))^{G_{K}} / \partial (\mu_n(R)) 
          = \partial \Gamma / \partial (\mu_n(R))$.
\end{Lemma}
\begin{Proof}
This is seen by taking Galois cohomology of the short exact sequence
(\ref{eqn-wd})
and recalling that $w_1$ is the composite of $w_*$ and an isomorphism
$H^1(K,\mu_n(\Rbar)) \isom R^\times/(R^\times)^n$.
\end{Proof}

In the case $n=p$ is prime
it is shown in \cite{DSS}, \cite{SchaeferStoll} 
that $w_1$ is injective. Then by Lemma~\ref{kerw1} 
each of the $\# \partial \Gamma$ possibilities
for $\rho$ represent the same element 
of $H/\partial R^ \times$. The lemma also gives a formula for 
$\# \partial \Gamma$. 
Writing $\dim$ for the dimension of a $\Z/p\Z$-vector space we have
\[ \dim  \partial \Gamma  = \dim \partial (\mu_p(R))  
  = \dim \mu_p(R) - \dim E(K)[p]. \]

For general $n$ we can use
Lemma~\ref{kerw1} to check whether $w_1$ is injective. 
As promised in \cite[\SSS3]{paperI}, 
we give an example to show that $w_1$ need not always be injective.

\begin{Example} Taking $n=4$ we consider the elliptic
curve $E/\Q$ with Weierstra{\ss} equation \[y^2 = x^3 + x + 2/13.\] 
A calculation using division polynomials shows that $[\Q(E[4]):\Q] = 48$. 
Hence $G \subset \GL_2(\Z/4\Z)$ is a subgroup of index $2$.
There are precisely three subgroups of index $2$ in $\GL_2(\Z/4\Z)$.
These may be viewed as kernels of $1$-dimensional characters, one of 
which factors via the determinant and another via the natural map 
to $\GL_2(\Z/2\Z) \isom S_3$. 
Thus the $3$ possibilities for $G$ correspond to whether 
$-1$, $\Delta_E$ or $-\Delta_E$ is a rational
square. 
In this example $\Delta_E = -(112/13)^2$. Computing $\Gamma$ from $G$ 
we find $\# \Gamma = 2^8$. It follows by Lemma~\ref{partialgamma}(i) 
that $\#\partial \Gamma = 2^4$. 
The points of order $2$ and $4$ on
$E$ each form a single Galois orbit, and their fields of definition
do not contain $\sqrt{-1}$. 
Hence $\#E(\Q)[4] = 1$ and $\#\partial(\mu_4(R)) = \#\mu_4(R) = 2^3$. 
It follows by Lemma~\ref{kerw1} that $w_1$ has kernel of order $2$.
\end{Example}

This example shows it would be difficult to do a $4$-descent directly
using the map $w_1$. Fortunately there are better ways of doing
$4$-descent: see the introduction for references.

%----------------------------------------------------------------------

\subsection{Initial equations for the covering curves}
\label{compcov}
Let $C_\xi \to E$ be the $n$-covering corresponding to some
$\xi \in H^1(K, E[n])$. 
In this section we are concerned with finding an explicit model 
for~$C_\xi$ as a genus one normal curve of degree $n^2$ in~$\PP^{n^2-1}$.

We explain first why we need to work with $\rho \in H$ 
representing~$\xi$, rather than with $\alpha \in R^\times$.
Recall the commutative diagram
\[ \xymatrix{ E(K) \ar[rr]^-{\delta} \ar[drr]_{\widetilde{F}}
               && H^1(K,E[n]) \ar[d]^{w_1} \\
               && R^\times/(R^\times)^n }
\]
from Proposition~\ref{Propw1}. The $K$-rational points on~$C_\xi$
should map to the points in~$E(K)$ whose image under~$\delta$
is~$\xi$, so whose image under $\widetilde{F} = w_1 \circ \delta$ is
$\alpha (R^\times)^n$.  This suggests defining the covering curve by
\begin{equation} \label{Calpha}
  \{(P, z) \in E \times \BA(R) : F(P) = \alpha z^n\} \,,
\end{equation}
where $\BA(R)$ denotes the affine space over~$K$ corresponding to the
$K$-vector space~$R$ (note that $\BA(R)$ was denoted $\mathcal R$
in~\cite{paperII}).  Working over~$\Kbar$ and writing $z_T$ for~$z(T)$,
which is the value at~$T$ of $z \in \Rbar$ considered as a map
on~$E[n]$, we note that the~$z_T$ can be used as a set of \coordinates{}
on~$\BA(R)$; hence the equation in~(\ref{Calpha}) may be written as
\[ F_T(P) = \alpha(T) z_T^n  \qquad \text{for all } T \in E[n] \,. \]
We see that for each point $P \in E \setminus E[n]$, there are
$n^2$ independent $n$th roots to take to obtain the~$z_T$. This makes
$n^{n^2}$ choices for~$z$, yet the covering map $C_\xi \to E$ 
has degree only~$n^2$.  The equation for $T = \origin$ reads
$1 = z_\origin^n$ (without loss of generality, $\alpha(\origin) = 1$), so
we can eliminate a factor of~$n$ by setting $z_\origin = 1$. Also, by
considering $z$ with $z_\origin = 1$ as a representative of a point in
the projective space $\PP(R)$ associated to~$\BA(R)$ and taking the
closure in $E \times \PP(R)$, we obtain a curve $C'$ that,
unlike~(\ref{Calpha}), maps surjectively to $E$ by the first projection. 
(We have `filled the gaps' above $P = \origin$ where the $F_T$ have 
a pole.) Now $C'$ is a projective curve covering~$E$ by a map of
degree~$n^{n^2-1}$. In the case $n = 2$, discussed further 
in \SSS\ref{twodesc} below, this curve splits into two
isomorphic components, and after projecting to~$\PP(R)$ we obtain an
intersection of two quadrics defining the desired curve~$C_{\xi}$.

If $n > 2$, then the curve~$C'$ defined above
splits into $n^{n^2-3}$ geometric components.
To see this, we may work over an algebraically closed field~$K$; then
we can take $\alpha = 1$.  We obtain an embedding
\[ \iota : E \To C', \qquad P \longmapsto \bigl(nP,
\widetilde{G}(P)\bigr) \] where $\widetilde{G} : E \to \PP(R)$ is
induced by $G : E \setminus E[n] \to \BA(R)$ (recall that $F(nP) =
G(P)^n$).  The image of $\iota$
is an $n$-covering of~$E$ by projection onto the
first factor, and the action of~$S \in E[n]$ on it is given by $z_T
\mapsto e_n(S,T) z_T$. In addition, $\mu_n(\Rbar)$ acts on~$C'$ in an
obvious way which is compatible with the action of~$E[n]$ and the map
$w : E[n] \to \mu_n(\Rbar)$. Taking into account that $\mu_n =
\mu_n(\Kbar) \subset \mu_n(\Rbar)$ acts trivially on~$\PP(R)$, this
shows that the components of~$C'$ are parametrised by
$\mu_n(\Rbar)/(w(E[n]) \mu_n)$.
Over arbitrary fields~$K$, and with~$\alpha$ representing 
any $\xi \in H^1(K, E[n])$, it can be shown (by an argument
similar to that in \cite[Section 3]{paperII}) that $C_\xi$ 
is a component of $C'$. The set of geometric components of~$C'$
is then isomorphic (as a set with Galois action) to the
Galois module $\mu_n(\Rbar)/(w(E[n]) \mu_n)$. In some cases,
this module has only one element defined over~$K$, so there is only
one component of~$C'$ that is defined over~$K$, which must be the one
we seek. But in general, there can be a large number of $K$-rational
components; and in any case, the equations one obtains (having degree
$n>2$) are not of the desired form, and it seems rather difficult to
pick out the relevant component.

Instead, we work with $\rho \in H$ representing~$\xi$. The analogue of the
above diagram is provided by Proposition~\ref{Propw2}:
\[ \xymatrix{ E(K) \ar[rr]^-{\delta} \ar[drr]_{\widetilde{r}}
               && H^1(K,E[n]) \ar[d]^{w_2} \\
               && H/\partial R^\times } \]
We can now define our covering curve using the relation $r(P) = \rho \partial z$.
Setting $z_\origin = 1$ as above, and then homogenising, this reads as
\begin{equation} \label{Crho}
  C_\rho = \{(P, z) \in E \times \PP(R)
                : r(P) z_\origin \Delta(z) = \rho \cdot (z \otimes z)\}
         \subset E \times \PP(R)
\end{equation}
with the covering map $C_\rho \to E$ given by projection to the first factor.
Note that the equation in~\eqref{Crho} is quadratic in~$z$. 
Over~$\Kbar$, in terms of the
\coordinates{}~$z_T$, the equations read
\[ r_{T_1,T_2}(P) z_\origin z_{T_1+T_2} = \rho(T_1,T_2) z_{T_1} z_{T_2} \]
with $T_1, T_2 \in E[n]$.
It is shown in \cite{paperII} that projecting to~$\PP(R)$ (which is equivalent
to eliminating $P \in E$) gives $n^2(n^2-3)/2$ linearly independent quadrics
defining $C_\rho \subset \PP(R) \isom \PP^{n^2-1}$ as a genus one normal curve
of degree $n^2$.

The equations for $C_\rho \subset \PP(R)$ are obtained from
\begin{equation} \label{quads1}
  (X-x_T) z_\origin^2 - \rho(T,-T) z_T z_{-T}
\end{equation}
for $T \in E[n] \setminus\{\origin\}$ and
\begin{equation} \label{quads2}
  (\Lambda_T - \lambda(T_1,T_2)) z_\origin z_T - \rho(T_1,T_2) z_{T_1} z_{T_2}
\end{equation}
for $T_1,T_2,T \in E[n] \setminus\{\origin\}$ with $T_1+T_2 = T$,
by taking differences to eliminate the indeterminates $X$ and $\Lambda_T$.
In fact the equations recorded in
\cite[Proposition 3.7]{paperII} are these differences.

%----------------------------------------------------------------------

\subsection{Improved equations for the covering curves}
\label{improve_eqns}
Suppose $\rho \in H$ represents 
an element~$\xi \in H^1(K, E[n])$ with trivial obstruction
(as defined in~\cite{paperI,paperII}), for example a Selmer group
element~$\xi \in \Sel^{(n)}(E/K)$. In the previous section we
showed how to write the covering curve $C_\rho$ as a curve
of degree~$n^2$ in~$\PP^{n^2-1}$.
We now want to write it as a curve of degree $n$
in~$\PP^{n-1}$. We recall how to do this using the Segre embedding method,
as described in \cite[\SSS5.3]{paperI} and~\cite{paperII}. 

First we fix the scaling of the $F_T$ so that each has leading coefficient
$1$ when expanded as a power series in the local parameter $x/y$ at $\origin$.
Then we define $\eps \in (R \otimes_K R)^\times$ by
\begin{equation} \label{defeps}
  \eps(T_1,T_2) = \frac {F_{T_1+T_2}(P)}{F_{T_1}(P) F_{T_2}(P-T_1)} 
\end{equation}
(which is independent of $P \in E$), compare Step~2 on p.~154 in~\cite{paperI}. 
Let $*_{\eps \rho}$ be the new multiplication on $R$ defined by
\[  z_1 *_{\eps \rho} z_2 = \Tr (\eps \rho \cdot (z_1 \otimes z_2)) \,. \]
Then the obstruction algebra
$A_\rho = (R,+,*_{\eps \rho})$ is a central simple algebra over~$K$
of dimension~$n^2$. 
Since we are assuming that $\rho$ represents an element with trivial obstruction
we have
$A_\rho \isom \Mat_n(K)$. In \SSS\ref{blackbox} we discuss how to
find such an isomorphism explicitly.

Recall that we have equations for $C_\rho \subset \PP(R)$. Since $R$ and~$A_\rho$
have the same underlying vector space, and we have now trivialised
the obstruction algebra, we get $C_\rho \subset \PP(\Mat_n)$. We project
$C_\rho$ away from the identity matrix onto the hyperplane of
trace zero matrices. 
As shown in~\cite{paperII}, the result is a curve $\widetilde{C}$
lying in the locus of rank~1 matrices. In other words the inclusion 
of this curve in $\PP(\Mat_n)$ factors via the Segre embedding
\[ \PP^{n-1} \times (\PP^{n-1})^\vee \to \PP(\Mat_n) \,. \]
Projecting onto a row or column gives 
either the degree-$n$ curve $C \to \PP^{n-1}$ we are looking for, 
or its dual $C \to (\PP^{n-1})^\vee$, which is a curve of degree~$n^2-n$.

Writing $z \in R = K \times L$ as $z = (z_\origin, z')$, the
projection to the subspace of trace-zero matrices corresponds to
eliminating~$z_\origin$ from the equations (recall that $\Tr(M_T) = 0$
for $T \neq \origin$, where $M_T$ gives the action of~$T$ on the
ambient space~$\PP^{n-1}$ of~$C$).  We note that if $T_1 + T_2 = T'_1
+ T'_2 = T$ and $\{T_1,T_2\} \neq \{T_1',T_2'\}$ then
$\lambda(T_1,T_2) \neq \lambda(T'_1,T'_2)$.  Assuming $n \ge 3$, it is
clear by~\eqref{quads1} and~\eqref{quads2} that eliminating
$z_\origin$ by linear algebra will reduce the dimension of the vector
space of quadrics by exactly~$n^2$. So after trivialising the algebra
we have $n^2(n^2-5)/2$ quadrics that are a basis for the space of
quadrics vanishing on
\[ \widetilde{C} \subset \PP({\Tr=0}) \isom \PP^{n^2-2} \,. \]
Together with the quadrics that are a product of a linear form
and the trace form, these span the space of quadrics vanishing on 
\[ \widetilde{C}  \subset \PP(\Mat_n) \isom \PP^{n^2-1}. \]
\begin{Lemma}
\label{lem:elim}
  Let $C \subset \PP^{n-1}$ be a genus one normal curve
  with homogeneous ideal $I(C) \subset K[x_1, \ldots, x_n]$. 
  Let $\widetilde{C}$ be the image of the map 
  $C \to \PP^{n-1} \times (\PP^{n-1})^\vee \to \PP(\Mat_n)$
  with homogeneous ideal
  $I(\widetilde{C}) \subset K[z_{11}, z_{12}, \ldots, z_{nn}]$. 
  If $f \in K[x_1,\ldots,x_n]$ is a homogeneous form, then
  \[ f(x_1, \ldots, x_n) \in I(C)
      \iff f(z_{11},z_{21}, \ldots,z_{n1}) \in I(\widetilde{C}) \,.
  \]
\end{Lemma}
\begin{Proof}
  This is clear since the dual curve spans $(\PP^{n-1})^\vee$.
\end{Proof}

The quadrics vanishing on $C \subset \PP^{n-1}$ may now be computed by
linear algebra. If $n \ge 4$ then these quadrics define $C$.  (In fact
they generate the homogeneous ideal.) When $n = 3$ 
the equation for $C$ is a ternary cubic. In \SSS\ref{threedesc} we explain
how this too may be computed using linear algebra.

%=========================================================================

\section{Application to 2-descent}
\label{twodesc}

We show that the method sketched above reduces in the case $n=2$ to
classical 2-descent, as described, for example, in \cite{CaL}, 
\cite{Schaefer}, and with algorithmic details in \cite{Simon_ell_paper}. 

Let $E$ be given by a short Weierstra{\ss} equation:
\[ E : \quad y^2 = f(x) = (x-e_1)(x-e_2)(x-e_3), \]
with $e_i \in \Kbar$. 
 Let $T_i=(e_i,0)$.
We have $R = K \times L$ where $L = K[e]$ is the cubic $K$-algebra
generated by $e$ with minimal polynomial $f(x)$. Let $\alpha \in
R^\times$ represent a Selmer group element. We write $\alpha_i =
\alpha(T_i)$. Without loss of generality the $K$-component of~$\alpha$
is~$1$, and we may regard~$\alpha$ as an element of~$L^\times$.  
It is well known (see \cite[Theorem 1.1]{Schaefer})
that $w_1$ induces an isomorphism
\[ H^1(K,E[2]) \isom \ker \big( L^\times/(L^\times)^2 \stackrel{N_{L/K}}{\ra}
K^\times/(K^\times)^2 \big). \] 
Therefore $\alpha$ has square norm, say 
$N_{L/K}\alpha = \alpha_1 \alpha_2 \alpha_3 = b^2$ for some $b \in K^\times$.  
Taking $F_{T_i} = x - e_i$ in (\ref{Calpha}), and setting $z_\origin = 1$,
we obtain equations
\begin{equation}
\label{twocov}
\begin{aligned}
  x - e_i &= \alpha_i z_i^2 \qquad \quad \text{ for } i = 1,2,3 \\
y &= \pm b z_1 z_2 z_3 
\end{aligned}
\end{equation}
where $z_i = z(T_i)$. Alternatively, since
\[ r(T_i,T_j) = \left\{ \begin{array}{ll} x - e_i & \text{ if } i=j \\
    y/(x-e_k) & \text{ if } \{i,j,k\} = \{1,2,3\}, \end{array} \right.
\] and $\alpha \in R^\times$ corresponds to $\rho \in H$ where
\[ \rho(T_i,T_j) = \left\{ \begin{array}{ll} \alpha_i & \text{ if } i=j \\
b/\alpha_k & \text{ if } \{i,j,k\} = \{1,2,3\}, \end{array} \right. \]
we obtain the same equations using~(\ref{Crho}). The components of $r$ 
and $\rho$ where one of the torsion points is $\origin$ are all trivial.
Notice that switching the sign of $b$ multiplies $\rho$ by $\partial
\gamma$ where $\gamma =(1,-1) \in K \times L$.

The first three equations in ~(\ref{twocov}) may be written (after
homogenisation) as
\[ x-eu_3^2 = \alpha u^2, \]
where $u=u_0+u_1e+u_2e^2$ is an ``unknown'' element of $L^\times$.
Expanding, eliminating~$x$, and equating coefficients of powers of~$e$ gives two
quadrics in $u_0,u_1,u_2,u_3$, defined over~$K$, which define
$C_\rho \subset \PP^3$ as a curve of degree $4$. We would like to
write $C_\rho$ as a double cover of $\PP^1$. 
The classical approach is to observe that one 
of the quadrics does not involve~$u_3$
and hence defines a conic~$S$ in $\PP^2(u_0,u_1,u_2)$; the projection
to~$S$ is a double cover~$C_\rho\to S$.  
If $\alpha$ is a Selmer group element then $S\isom\PP^1$.
In practice one expresses the isomorphism $\PP^1\to S$ as a parametrisation
$u_j=q_j(v_0,v_1)$ for $j=1,2,3$, where the $q_j$ are binary
quadratics; substituting into the first quadric in the~$u_j$, in which
the only term involving $u_3$ is (a non-zero constant times) $u_3^2$,
we find an equation for $C_\rho$ of the form $u_3^2=Q(v_0,v_1)$ where
$Q$ is a binary quartic.

To obtain the $2$-covering map $C_\rho\to E$ we simply substitute
in~(\ref{twocov}) to recover~$x$, while $y$ is determined up to sign
by $y^2=N_{L/K}(\alpha u^2)$.  Hence we have two possibilities for the
covering map, which differ by negation on~$E$; these $2$-coverings are
equivalent.

We now compare with the Segre embedding method as described in 
Section~\ref{improve_eqns}.
The obstruction algebra is $A = (R,+,*_{\eps \rho})$ where 
using~(\ref{defeps}) we compute
\[ \eps(T_i,T_j) = \left\{ \begin{array}{ll} 1/f'(e_i) & \text{ if } i=j \\
1/(e_i-e_j) & \text{ otherwise. } \end{array} \right. \]
Let $S = \{ z \in \PP(A) : \Trd(z) = \Nrd(z) = 0 \}$. Our general
recipe says that if we project $C_\rho$ to the plane $\{ \Trd(z) =
0\}$, then the result lies in $S$. This gives $C_\rho$ as a double
cover of $S$. In fact, $S$ is defined by $z_\origin=0$ and $z*_{\eps \rho}z
=0$. The latter works out as
\[ \sum_{i=1}^3 \frac{\alpha_i}{f'(e_i)} z_i^2 = 0 \]
which is one of the quadrics in the pencil defining $C_\rho$. Hence
our conic $S$ is the same as that considered in the classical approach. 
The problems of trivialising $A$ and finding a rational point on $S$
are clearly equivalent.   
Once we have trivialised $A$ we get an isomorphism $S \isom \PP^1$
by projecting to a row or column (it does not matter which, since $\PP^1$
is self-dual). Exactly as before we can then write 
$C_{\rho}$ as a double cover of $\PP^1$.

%=========================================================================

\section{Application to 3-descent}
\label{threedesc}

We give further details of our algorithms in the case $n=3$.
Let $E$ be an elliptic curve over a number field $K$.  We fix an
isomorphism $E[3] \isom (\Z/3\Z)^2$, say $T_{ij} \mapsto (i,j)$, and let
$T = T_{11}$. 
We assume that the Galois action on $E[3]$ is generic, in the sense that
$\rho_{E,3} : \GK \to \GL_2(\Z/3\Z)$ is surjective\footnote{If $K= \Q$ 
then there are exactly $8$ possibilities for $\im(\rho_{E,3})$
up to conjugacy. Our \Magma\ implementation relies on a similar
analysis of all $8$ cases.}.
Then there is a tower of number fields
\[ \xymatrix{ \hspace{4.5em} M = K(E[3]) \ar@{-}[d]^2 \\ \,\, M^+ \ar@{-}[d]^3 
 \\ \hspace{3.5em} L = K(T) \ar@{-}[d]^2 \\ \,\, L^+ \ar@{-}[d]^4 \\ K }  \]
where $M^+$ is the subfield of $M$ fixed by $T_{ij} \mapsto T_{ji}$ 
and $L^+$ is the subfield of $L$ fixed by $\sigma: T \mapsto -T$. We write
$\iota_{ij} : L \to M$ for the embedding given by $T \mapsto T_{ij}$. 
Thus $\iota_{11}$ is the natural inclusion and $\iota_{ij} \circ \sigma = 
\iota_{-i,-j}$.

There are two orbits for the action of $\GK$ on $E[3]$, with
representatives $\origin$ and $T$, and six orbits for the action of
$\GK$ on $E[3] \times E[3]$, with representatives 
\[  (\origin,\origin), \,\, (T,\origin), \,\,(\origin,T), \,\,(-T,-T),
 \,\,(T,-T), \,\,(T_{10},T_{01})  \]
chosen so that each pair sums to either $\origin$ or $T$.
Using these representatives we identify $R = K \times L$ and
(writing $r = (r_1,r_2)$, $s = (s_1,s_2)$ with $r_1, s_1 \in K$, $r_2, s_2 \in L$)
\begin{equation}
\label{tensorformula}
\begin{aligned}
R \otimes_K R & \isom K \times L \times L \times L \times L \times M \\
r \otimes s & \mapsto \left(r_1 s_1, r_2 s_1, r_1 s_2, 
\sigma(r_2) \sigma(s_2), r_2 \sigma(s_2), \iota_{10}(r_2) \iota_{01}(s_2) 
\right).
\end{aligned}
\end{equation}
The comultiplication $\Delta : R \to R \otimes_K R$ is given by 
\[ (r_1, r_2) \mapsto (r_1, r_2, r_2, r_2, r_1, r_2) \]
and the trace map $\Tr: R \otimes_K R \to R$ by
\begin{equation}
\label{traceformula}
  (a,b_1,b_2,b_3,b_4,c) \mapsto \left(a + \Tr_{L/K}(b_4) ,
b_1 + b_2 + b_3 + \Tr_{M/L}(c) \right). 
\end{equation}

In \SSS\ref{sec:w1} we showed how to compute $\alpha =(1,a) \in
R^\times$ representing an element of $\Sel^{(3)}(E/K)$.  We now
compute
\[  u = \sqrt[3]{ a \sigma(a)}, \qquad v= \sqrt[3]{\iota_{10}(a)
\iota_{01}(a)/a}, \]
by extracting cube roots in $L^+$ and $M^+$. Since $\det \rho_{E,3}$ is
the cyclotomic character, these fields have
no non-trivial cube roots of unity. Hence $u$ and $v$ are uniquely 
determined. The elements $\eps$ and $\rho$ in $R \otimes_K R$ are defined by
\begin{equation}
\label{epsrho}
\begin{aligned}
\eps &= (1,1,1,1,1,e_3(T_{10},T_{01})) \\
\rho &= (1,1,1,\sigma(a)/u,u,v) 
\end{aligned}
\end{equation}
where $e_3 : E[3] \times E[3] \to \mu_3$ is the Weil pairing.
The reader is warned that this $\eps$ is different from the one given
in~(\ref{defeps}). We explain how to correct for this in 
\SSS\ref{exteqns}.
The sign convention we use for the Weil pairing does matter,
but is not worth fixing here since we can correct for it 
later if necessary.

Let $u_1, \ldots, u_8$ be a basis for $L$ over $K$. (In 
\SSS\ref{compobs}
we describe how to make a good choice of basis.) Then $R$ has
basis $r_1, \ldots ,r_9$ where $r_1 = (1,0)$ and $r_{i+1} = (0,u_i)$.
Structure constants $c_{ijk} \in K$ for the obstruction algebra 
$A = (R,+,*_{\eps \rho})$ are now given by
\[ \Tr (\eps \rho \cdot (r_i \otimes r_j)) = \textstyle\sum_{k=1}^9
c_{ijk} r_k. \] The $c_{ijk}$ are computed using the formulae
(\ref{tensorformula}), (\ref{traceformula}) and (\ref{epsrho}).  Since
$\alpha$ represents a Selmer group element we know that 
$A_\rho \isom \Mat_3(K)$.
In \SSS\ref{blackbox} we show how to find
such an isomorphism explicitly. In other words, we find (non-zero)
matrices $M_1, \ldots, M_9 \in \Mat_3(K)$ satisfying
\begin{equation}
\label{trivmats}
  M_i M_j = \textstyle\sum_{k=1}^9 c_{ijk} M_k. 
\end{equation}

We fix a Weierstra{\ss} equation $y^2 = x^3 + a x + b$
for $E$ and let $T=(x_T,y_T)$. The tangent line to $E$ at $T$
has slope $\lambda_T = \lambda(T,T) = (3 x_T^2 + a)/(2y_T)$.
We define linear forms in indeterminates $z_1, \ldots, z_8$,
\begin{align*}
z_T &=  \textstyle\sum_{i=1}^8 u_i z_i & 
z_{10} &=  \textstyle\sum_{i=1}^8 \iota_{10}(u_i) z_i \\
z_{-T} &=  \textstyle\sum_{i=1}^8 \sigma(u_i) z_i &
z_{01} &=  \textstyle\sum_{i=1}^8 \iota_{01}(u_i) z_i 
\end{align*} 
where $u_1, \ldots, u_8$ is our basis for $L$ over $K$. Let 
 $Q_1$ and $Q_2$ be the quadrics with coefficients in $L^+$ 
and $M^+$ defined by
\begin{align*}
Q_1(z_0, \ldots,z_8) &=  x_T z_0^2 + \rho_5 z_T z_{-T} \\
Q_2(z_0, \ldots,z_8) &= (\lambda_T + \kappa_T) 
z_0 z_T  - \rho_4 z_{-T}^2 + \rho_6 z_{10} z_{01} 
\end{align*}
where the $\rho_i$ are the components of $\rho$, and
$\kappa_T = \frac{1}{3}(\iota_{10}(\lambda_T) + 
\iota_{01}(\lambda_T) - \lambda_T)$. Writing 
each coefficient in terms of fixed $K$-bases for $L^+$ and $M^+$,
we obtain $[L^+:K]=4$ quadrics from $Q_1$ and 
$[M^+:K]=24$ quadrics from $Q_2$. In
\cite{paperI} these are called the quadrics of types I and II. 
We choose our basis
for $L^+$ so that its first element is $1$, and ignore
the first type I quadric. The result is 27 quadrics in 
$K[z_0, \ldots, z_8]$. 

\begin{Lemma} These $27$ quadrics generate the homogeneous ideal
of the degree~$9$ curve $C_\rho \subset \PP(R) = \PP^8$.
\end{Lemma}
\begin{Proof}
Let $v_1=1,v_2,v_3,v_4$ be a basis for $L^+$ over $K$. We write
\[ x_T z_0^2 + \rho(T,-T) z_{T} z_{-T} = \textstyle\sum_{i=1}^4 v_i
q_i(z_0,\ldots,z_8) \]
where $q_1, \dots,q_4 \in K[z_0, \ldots,z_8]$. 
Then~(\ref{quads1}) becomes
\begin{align*}
X &= q_1(z_0, \ldots, z_8) \\
0 &= q_i(z_0, \ldots, z_8) \qquad \text{ for } i = 2,3,4. 
\end{align*}
We eliminate $X$ by ignoring the first quadric $q_1$.

Next we take $(T_1,T_2) =(T_{10},T_{01})$
and $(-T,-T)$ in~(\ref{quads2}). Subtracting to eliminate $\Lambda_T$ 
gives the quadric
\[ (\lambda(T_{10},T_{01}) - \lambda(-T,-T) ) z_0 z_T - \rho(-T,-T) z_{-T}^2
+ \rho(T_{10},T_{01}) z_{10} z_{01}. \]
Assuming that 
\begin{equation}
\label{eqnu}
\lambda(T_{10},T_{01}) = \tfrac{1}{3} 
\big( \lambda(T_{10},T_{10}) +\lambda(T_{01},T_{01}) - \lambda(T,T) \big)
\end{equation}
this is precisely the quadric $Q_2$.
To complete the proof we note that~(\ref{eqnu}) is a special case 
of the following lemma.
\end{Proof}

\begin{Lemma}
\label{slopes}
Let $T_1,T_2,T_3 \in E[3] \setminus \{\origin\}$ with $T_1 + T_2 + T_3 = \origin$.
Then 
\[ \lambda(T_1,T_2) = \tfrac{1}{3} \sum_{i=1}^3 \lambda(T_i,T_i). \]
\end{Lemma}
\begin{Proof}
Let $f_i = y - \lambda_i x - \nu_i$ be the equation of the tangent line at $T_i$
and $f = y - \lambda x -\nu$ the equation of the 
chord through $T_1$, $T_2$ and $T_3$. As rational functions 
on $E$ we have $f_1 f_2 f_3 = f^3$. Expanding as
power series in the local parameter $x/y$ at $\origin$ it follows that
$\lambda = \tfrac{1}{3}(\lambda_1 + \lambda_2 + \lambda_3)$ as required.
\end{Proof}

The remainder of the algorithm is the same for all Galois actions
on $E[3]$. As specified in \SSS\ref{improve_eqns},
we intersect the above space of quadrics with $K[z_1, \ldots, z_8]$
to leave an 18-dimensional space of quadrics defining the projection
of $C_\rho \subset \PP(R) = \PP^8$ to $\PP(L) = \PP^7$. In other words
we eliminate the monomials $z_0z_i$ by linear algebra.
We then make the following changes of \coordinates{}
\begin{itemize}
\item A change of \coordinates{} corresponding to pointwise 
multiplication by the ``fudge factor'' $1/y_T \in L$ (relative to the 
basis $u_1,\ldots,u_8$). This is to make up for the fact
that the definitions of $\eps$ in~(\ref{defeps}) and~(\ref{epsrho}) 
are different. We explain this further in \SSS\ref{exteqns}.
\item A change of \coordinates{} corresponding to the trivialisation
of the obstruction algebra. 
\end{itemize}
We now have 18 quadrics in variables $z_{ij}$ where $1 \le i,j \le 3$.
These \coordinates{} correspond to the standard basis for
$\Mat_3(K)$.  

\begin{Lemma}
\label{bi-subst}
Substituting $z_{ij} = x_i y_j$ gives a basis for 
the space of $(2,2)$-forms vanishing on the image of 
$C \to \PP^2 \times (\PP^2)^\vee$. 
\end{Lemma}
\begin{Proof}
We start with a basis for the 18-dimensional space of quadrics vanishing 
on $\widetilde{C} \subset \PP(\Tr=0)$. This may be identified with 
the space of quadrics vanishing on 
$\widetilde{C} \subset \PP(\Mat_3)$ that are ``singular at $I_3$''.
(A quadric is ``singular at $I_3$'' if when we write it relative to 
a basis for $\Mat_3(K)$ with first basis vector $I_3$, the first variable 
does not appear.)
Substituting $z_{ij} = x_i y_j$ gives a surjective linear map $\Phi$
from the 45-dimensional space of quadrics in $z_{11}, \ldots, z_{33}$ to the
36-dimensional space of $(2,2)$-forms in $x_1,x_2,x_3$ and $y_1,y_2,y_3$.
The kernel is spanned by the $2 \times 2$ minors of the matrix
$(z_{ij})$ and is a complement to the space of 
quadrics ``singular at $I_3$''.
Thus $\Phi$ induces an isomorphism between the space of quadrics 
vanishing on $\widetilde{C} \subset \PP(\Tr=0)$ and the space
of $(2,2)$-forms vanishing on the image of $C \to \PP^2 \times 
(\PP^2)^\vee$. 
\end{Proof}

We multiply each of the forms constructed in Lemma~\ref{bi-subst}
by the $x_i$ to obtain $54$ forms of \bidegree{} $(3,2)$. The following
lemma shows that there is a ternary cubic $f$, unique up to scalars,
such that $y_1^2 f(x_1,x_2,x_3)$ belongs to the span of these
$54$ forms. Moreover $f$ is the equation of the curve $C \subset \PP^2$
we are looking for.

\begin{Lemma}
\label{getcubic} 
Let $C \subset \PP^2$ be a non-singular plane cubic with equation
$f=0$. Let $V$ be the space of $(2,2)$-forms vanishing 
on the image of $C \to \PP^2 \times (\PP^2)^\vee$. Then 
$y_1^2 f(x_1,x_2,x_3)$ is a $(3,2)$-form in the ideal 
generated by $V$. Moreover this is the only such polynomial 
up to scalars.
\end{Lemma}

\begin{Proof}
By Euler's identity $3f = \sum x_i 
\frac{\partial f}{\partial x_i}$ we have
\begin{equation}
\label{tc-id}
3 y^2_1 f  = x_2 y_1 g_{12} + x_3 y_1 g_{13} 
+  \tfrac{\partial f}{\partial x_1} y_1 \ell. 
\end{equation}
where $\ell = \sum_{i=1}^3 x_i y_i$ and 
$g_{ij} = y_i \tfrac{\partial f}{\partial x_j} - 
y_j \tfrac{\partial f}{\partial x_i}$ are \bihomogeneous\ forms
vanishing on the image of $C \to \PP^2 \times (\PP^2)^\vee$.
Exactly as in the proof of Lemma~\ref{lem:elim},
the uniqueness statement follows from the fact that
the dual curve spans $(\PP^{2})^\vee$.
\end{Proof}

Lemma~\ref{getcubic} allows us to compute the ternary cubic
$f$ by linear algebra.  If we had made the wrong choice of sign for
the Weil pairing, then the matrices $M_i$ in~(\ref{trivmats}) would be
the transposes of the desired ones; switching the roles of the $x_i$
and $y_j$ corrects for this.

Our implementation in \Magma\ 
over $K= \Q$ finishes by minimising and reducing
the ternary cubic as described in \cite{minred234}. The covering map, 
computed using the classical formulae in \cite{AKM3P}, is also returned.

%=========================================================================

\section{A good basis for the obstruction algebra}
\label{compobs}

The obstruction algebra $A_\rho = (R,+,*_{\eps \rho})$ was defined
in \SSS\ref{improve_eqns}.
In the case $K = \Q$ we explain how to
choose a $\Q$-basis for $R$ so that the structure constants for
$A_\rho$ are small integers. This is useful for the later parts of our
algorithm, for example when trivialising the obstruction algebra as
described in the next section.

We recall that $R$ is a product of number fields. Its ring of integers
$\OO_R$ is the product of the rings of integers of these fields.
A fractional ideal in $R$ is just a tuple of fractional
ideals, one for each field, and a prime ideal is a tuple where 
one component is a prime ideal, and all other components 
are unit ideals.

Let $\alpha \in R^\times$ represent $w_1(\xi)$ for some 
$\xi \in H^1(\Q,E[n])$.
We write $(\alpha) = \bb \cc^n$ where $\bb$ is integral and $n$th 
power free. We then choose as our $\Q$-basis for $R$ a $\Z$-basis for
$\cc^{-1}$ that is LLL-reduced with respect to the inner product 
\begin{equation}
\label{iprod}
\langle z_1, z_2 \rangle = \sum_{T \in E[n]} |\alpha(T)|^{2/n} z_1(T)
\overline{z_2(T)}. 
\end{equation}
where the bar denotes complex conjugation.
In the remainder of this section we explain why this is a good choice.
Notice that in defining the inner product we have implicitly fixed 
an embedding $\Qbar \subset \C$.

We restrict to the case $n=2m-1$ is odd and take for $\eps$ the
square root of the Weil pairing, \ie $\eps(S,T) = e_n(S,T)^m$.
(This is also the choice we made in \SSS\ref{threedesc} for $n = 3$.
See \SSS\ref{exteqns} for a discussion of possible choices for~$\eps$
and their relation.)
By definition of $w_1$ (see \cite[\SSS3]{paperI}) there exists
$\gamma \in \Rbar^\times$ with $\gamma^n = \alpha$ and 
$w(\xi_\sigma) = \sigma(\gamma)/\gamma$ for all $\sigma \in \GQ$.
Then $w_2(\xi) = \rho \partial R^\times$
where $\rho = \partial \gamma \in (R \otimes R)^\times$.

\begin{Lemma} 
\label{lem:order}
The structure constants for $A_\rho$ with respect 
to a $\Z$-basis for $\cc^{-1}$ are integers, 
\ie $(\cc^{-1},+,*_{\eps \rho}) \subset A_\rho$ is an order.
\end{Lemma}
\begin{Proof}
Let $\pp$ be a prime of $R$.
Put $r = \ord_\pp(\bb)$ and $q = \ord_\pp(\cc)$ so that
$\ord_\pp(\alpha) = qn+r$ with $0 \le r<n$. Let 
$z_1,z_2 \in \cc^{-1}$. Then $\ord_\pp(z_i) \ge -q$ for $i=1,2$.
Extending $\ord_\pp$ to $\Rbar^\times$ and recalling that
$\gamma^n = \alpha$, we have $\ord_\pp(\gamma z_i) \ge 0$. Then
\[ \begin{array}{rcl}
z_1 *_{\eps \rho} z_2 & = & \Tr(\eps \rho \cdot (z_1 \otimes z_2)) \\
 & = & \gamma^{-1} \Tr(\eps \cdot (\gamma z_1 \otimes \gamma z_2)). 
\end{array} \]
Since $\eps \in R \otimes_K R$ is integral and the trace map 
$\Tr : R \otimes_K R \to R$ preserves integrality we deduce
$\ord_\pp(z_1 *_{\eps \rho} z_2) \ge -(qn+r)/n$. Since this valuation
is an integer we must therefore have $\ord_\pp(z_1 *_{\eps \rho} z_2) \ge -q$.
Repeating for all primes $\pp$ of $R$ 
it follows that $z_1 *_{\eps \rho} z_2 \in \cc^{-1}$ as required.
\end{Proof}

Let $\tau \in \GQ$ be complex conjugation. (Recall that we fixed an embedding 
$\Qbar \subset \C$.) Since $n$ is odd we 
have $H^1(\R,E[n]) = 0$ and so $\tau(\gamma)/\gamma = w(\xi_\tau) 
= w(\tau(S)-S)$ for some $S \in E[n]$. Therefore dividing $\gamma$ by 
$w(S)$ we may assume that $\gamma : E[n] \to \Qbar$
is $\GR$-equivariant.
It follows by \cite[Lemma 4.6]{paperI} that pointwise multiplication 
by $\gamma$ defines an 
isomorphism $A_{\rho} \otimes \R \isom A_{1} \otimes \R$.

Let $T_1$, $T_2$ be a basis for $E[n](\C)$ with $\overline{T}_1 = T_1$,
$\overline{T}_2 = -T_2$ and $e_n(T_1,T_2) = \zeta_n$. We define
$$ h(T_1) = \begin{pmatrix}
    0   &  1 & 0  &  \cdots &       0     \\ 
    0   &  0 & 1  &  \cdots &       0     \\
 \vdots & \vdots      & \vdots & &  \vdots   \\
    0   &  0  & 0 & \cdots    &      1    \\ 
    1   &  0  & 0 & \cdots    &      0       
\end{pmatrix}, \quad 
h(T_2) = \begin{pmatrix}
    1   &    0   &    0    & \cdots &      0     \\
    0   & \zeta_n  &    0    & \cdots &      0     \\
    0   &    0   & \zeta_n^2 & \cdots &      0     \\ 
 \vdots & \vdots & \vdots  &        &   \vdots   \\
    0   &   0    &    0    & \cdots & \zeta_n^{n-1}  
\end{pmatrix}, $$
and
\[ \begin{array}{crcl}
h : & E[n](\C) & \to & \Mat_n(\C) \\
& r T_1 + s T_2 & \mapsto & \zeta_n^{-rs/2} h(T_1)^r h(T_2)^s 
\end{array} \]
where the exponent of $\zeta_n$ is an element of $\Z/n\Z$.
It may be verified that \[h(S) h(T) = \eps(S,T) h(S+T)\] for all
$S,T \in E[n]$. Hence there is an isomorphism 
$A_1 \otimes \C \isom \Mat_n(\C)$ given by $z \mapsto \sum_T z(T) h(T)$. 
Since this isomorphism respects complex conjugation 
it restricts to an isomorphism $A_1 \otimes \R \isom \Mat_n(\R)$.

Composing the isomorphisms defined in the previous two paragraphs gives a 
trivialisation of $A_\rho$ over $\R$, \ie
\begin{equation}
\label{realtriv}
  A_\rho \otimes \R  \isom  \Mat_n(\R) \, ; \quad 
z \mapsto \sum_{T \in E[n]} \gamma(T) z(T) h(T).    
\end{equation} 
We use this trivialisation first to compute the discriminant of the 
order in Lemma~\ref{lem:order} and then to explain why we chose the
inner product~(\ref{iprod}). 
The \emph{discriminant} $\Disc (R)$ of $R$ is the product 
of the discriminants of the constituent fields. The \emph{norm} of 
$\bb \subset \OO_R$ is $\Norm \bb = \# (\OO_R/\bb)$.

\begin{Lemma} 
\label{compdisc}
The order $\O = (\cc^{-1},+,*_{\eps \rho}) \subset A_\rho$ 
has discriminant \[n^{n^2} (\Norm \bb)^{2/n} \Disc(R).\] 
\end{Lemma}
\begin{Proof}
Let $r_1, \ldots , r_{n^2}$ be a $\Z$-basis for $\cc^{-1}$ mapping to
matrices $M_1, \ldots , M_{n^2}$ (say) under the 
trivialisation~(\ref{realtriv}). Then the discriminant of $\O$ is 
$\Disc(\O) = \det (\Trd(r_i r_j))= \det (\Tr(M_i M_j))$. But
$$ \Tr(h(S)h(T)) = \left\{ \begin{array}{ll} n & \text{ if } S+T = \origin \\
0 & \text{ otherwise.} \end{array} \right. $$
Noting that $[-1]$ is an even permutation of $E[n]$ we compute
$$ \begin{array}{rcl}
\Disc(\O) & = & \det\bigl(n \sum_{T \in E[n]} \gamma(T) r_i(T) 
\gamma(-T) r_j(-T)\bigr)_{i,j} \\
& = & n^{n^2} \bigl(\prod_{T \in E[n]} \gamma(T)^2 \bigr)
 \bigl(\det(r_i(T))_{i,T}\bigr)^2.
\end{array} $$
By considering the basis for $\Rbar = \Map(E[n],\Kbar)$ consisting 
of indicator functions it is clear that for $z \in R$ we have
$\Tr_{R/\Q}(z) = \sum_T z(T)$ and $N_{R/\Q}(z) = \prod_T z(T)$.
Since $\gamma$ is $\GR$-equivariant we also have 
$\prod_{T} \gamma(T) \in \R$. Hence $\prod_{T} \gamma(T)^2 
= | N_{R/\Q}(\alpha) |^{2/n}$ and
\[ \bigl(\det(r_i(T))_{i,T}\bigr)^2 
= \Disc(r_1, \ldots, r_{n^2}) = (\Norm \cc)^{-2} 
\Disc(R).\] Recalling that $(\alpha) = \bb \cc^n$ the result is now clear.
\end{Proof}

\begin{Remark}
If we start with a Selmer group element
then the discriminant computed in Lemma~\ref{compdisc} 
is a product of primes dividing $n$ and primes of bad 
reduction for $E$.
Indeed if $\pp$ is a prime of $R$ not dividing any of these primes
then $\ord_\pp(\alpha) \equiv 0 \pmod{n}$. The term $\Disc(R)$ 
is of the stated form by (the easier implication of) 
the criterion of N\'eron-Ogg-Shafarevich.
\end{Remark}

Next we give some justification for our choice of inner product.
(See also \SSS\ref{latred} and the examples in \SSS\ref{sec:examples}.)
\begin{Lemma}
\label{ip}
The real trivialisation~(\ref{realtriv}) identifies the inner 
product~(\ref{iprod}) with (a scalar multiple of) the standard 
Euclidean inner product $\langle~,~\rangle$ on $\Mat_n(\R) \isom \R^{n^2}$. 
\end{Lemma}
\begin{Proof}
Extending $\langle~,~\rangle$ to an inner product on 
$\Mat_n(\C)$ we have
\[ \langle h(S),h(T) \rangle = \left\{ \begin{array}{ll}
n & \text{ if } S=T \\
0 & \text{ otherwise. }
\end{array} \right. \]
Therefore if $z_1, z_2 \in R$ map to $M_1,M_2 \in \Mat_n(\R)$ then 
\[ \begin{array}{rcl}
\langle M_1 , M_2 \rangle & = & n \sum_{T \in E[n]} | \gamma(T) |^2 
z_1(T) \overline{z_2(T)} \\
& = & n \sum_{T \in E[n]} | \alpha(T) |^{2/n} 
z_1(T) \overline{z_2(T)}. 
\end{array} \]
\end{Proof}

In principle we could now bound the size of the structure
constants. But in practice the structure constants are much smaller
than these bounds would suggest. (We encounter a similar situation at
the end of \SSS\ref{blackbox}.)

%=========================================================================

\section{Inside the ``Black Box'': How to trivialise a matrix algebra}
\label{blackbox}

Our work on $n$-descent on elliptic curves requires us to make
the Hasse principle explicit.  In the case $n=2$ this means we
have to solve a conic in order to represent a $2$-Selmer group element as a
double cover of~$\PP^1$ rather than as an intersection of quadrics
in~$\PP^3$. In this section we discuss the general case, and in
particular give an algorithm that is practical when $K=\Q$ and $n=3$.
See~\cite{IvanyosRonyaiSchicho} for a complexity analysis of our
method and its natural generalisation to arbitrary $K$ and $n$.

\subsection{Central simple algebras}
We recall some standard theory. See for example \cite[Part II]{WeilBNT}.
Let $K$ be a field. A central simple algebra $A$ over $K$ is a
finite-dimensional algebra over $K$ with centre $K$ and no two-sided
ideals (except $0$ and $A$).  Wedderburn's Theorem states that $A$ is
then isomorphic to a matrix algebra over a division algebra (\ie skew
field) $D$ with centre $K$.  The Brauer group $\Br(K)$ of $K$ is the
set of equivalence classes of central simple algebras over $K$, where
$A$ and $A'$ are equivalent if they are matrix algebras over the same
division algebra $D$.  The group law is given by tensor product,
\ie $[A] \cdot [A'] = [A \otimes_K A']$, and the inverse of $[A]$ is
the class of the opposite algebra $A^{\operatorname{op}}$ obtained by
reversing the order of multiplication.  The identity element is the
class of matrix algebras over $K$.

If $A$ is a central simple algebra over $K$ and 
$L/K$ is any field extension then $A_L = A \otimes_K L$ is a
central simple algebra over $L$. The reduced trace and norm are
defined as $\Trd(a) = \tr(\varphi(a))$ and $\Nrd(a) = \det(\varphi(a))$
where $\varphi : A_{\Kbar} \isom \Mat_n(\Kbar)$ is an isomorphism
of $\Kbar$-algebras. These definitions are independent of the choice 
of $\varphi$ by the Noether-Skolem theorem. 
We likewise define the rank of $a \in A$ to be 
the rank of $\varphi(a)$.

Now let $K$ be a number field. For each place $v$ of $K$ there
is a natural map $\Br(K) \to \Br(K_v)$ given by 
$[A] \mapsto [A_{K_v}]$. We recall \cite[\SSS32]{Reiner}
that $A_{K_v}$ is a matrix algebra over $K_v$ for all $v$ outside
a finite set of places depending on $A$.
It is then one of the main results of class field theory 
that the map
\begin{equation}
\label{brmap}
 \Br(K) \to \bigoplus_{v \in M_K} \Br(K_v) 
\end{equation}
is injective. Explicitly, this says that a central simple algebra $A$
over $K$ can be \emph{trivialised} (\ie is isomorphic to a
matrix algebra over $K$) if and only it can be trivialised everywhere
locally.  In particular deciding whether a
central simple algebra over $K$ can be trivialised is essentially a local
problem, given some global information restricting the places
to consider to a finite set. This latter usually involves
factorisation.

\subsection{Statement of the problem}
The problem we address is rather different. Given a $K$-algebra
$A$ known to be isomorphic to $\Mat_n(K)$,
we would like to find such an isomorphism explicitly. More specifically,
we want a practical algorithm that takes as input a list of structure
constants $c_{ijk} \in K$, giving the multiplication on $A$
relative to a $K$-basis $\av_1, \ldots, \av_{n^2}$ by the rule
\[  \av_i \av_j  = \sum_{k} c_{ijk} \av_k, \]
and returns as output a basis
$M_1, \ldots ,M_{n^2}$ for $\Mat_n(K)$ satisfying
\begin{equation}
\label{strM}
  M_i M_j  = \sum_{k} c_{ijk} M_k. 
\end{equation}
The output is far from unique, as we are free to conjugate
the $M_i$ by any fixed matrix in $\GL_n(K)$.

\subsection{Zero-divisors}
\label{zerodiv}
Let $A$ be a central simple algebra of dimension $n^2$ 
over a field $K$. If $n$ is prime then by 
Wedderburn's theorem either $A \isom \Mat_n(K)$ or 
$A$ is a division algebra. In particular $A \isom \Mat_n(K)$ 
if and only if it contains a zero-divisor. 

Once we have found a zero-divisor it is easy to find a trivialisation
$A \isom \Mat_n(K)$. More generally (\ie dropping our assumption that
$n$ is prime) it is enough to find $x \in A$ of rank $r$ with
$(r,n)=1$. Indeed as $A$-modules we have $Ax \isom M^r$ and $A \isom
M^n$ where $M$ is the unique faithful simple module.  By taking
kernels (or cokernels) of sufficiently general $A$-linear maps we can
apply Euclid's algorithm to the dimensions and so explicitly compute
$M$.  Since the natural map $A \to \End_K(M) \isom \Mat_n(K)$ is an
isomorphism, this gives the required trivialisation of $A$.

\subsection{Maximal orders}
\label{maxorder}
Let $A$ be a central simple algebra of dimension $n^2$ over $\Q$.
An \emph{order} in $A$ is a subring $\OO \subset A$ whose additive group
is a free $\Z$-module of rank $n^2$. Thus a 
$\Q$-basis $a_1=1,a_2, \ldots, a_{n^2}$ for $A$ is a $\Z$-basis for 
an order $\OO$ if and only if the structure constants are 
integers. We can reduce to this case by clearing denominators.
The discriminant of $\OO$ is defined as
\[ \Disc(\OO) = | \det ( \Trd(a_i a_j) ) |. \]
A \emph{maximal order} $\OO \subset A$ is an order that is not 
a proper subring of any other order in $A$.
It is shown in \cite[\SSS25]{Reiner} that all maximal
orders in $A$ have the same discriminant, which we denote $\Disc(A)$. 
Moreover if $A_{\Q_p} \isom \Mat_{\kappa_p}(D_p)$ where $D_p$ is a division
algebra over $\Q_p$ with $[D_p:\Q_p]=m_p^2$ then 
\begin{equation}
\label{discA}
   \Disc(A) = (\prod_p p^{(m_p-1) \kappa_p})^n.  
\end{equation}
By the injectivity of~(\ref{brmap}) it follows that
$A \isom \Mat_n(\Q)$ if and only if $\Disc(A)=1$
and $A_\R \isom \Mat_n(\R)$. (In fact we can dispense with the 
real condition, in view of the description of the image 
of~(\ref{brmap}) also given by class field theory.)

It is well known that every maximal order in $\Mat_n(\Q)$ 
is conjugate to $\Mat_n(\Z)$. 
By computing a maximal order our original problem 
(in the case $K= \Q$) is reduced to the following:
given structure constants for a ring known to be 
isomorphic to $\Mat_n(\Z)$, find such an 
isomorphism explicitly.

\subsection{Lattice reduction}
\label{latred}
Let $L \subset \R^m$ be a lattice spanned by the rows of an $m$ by $m$
matrix $B$. Then $\det L = |\det B|$ depends only on $L$ and not on $B$.
By the geometry of numbers, $L$ contains a non-zero vector $x$ with
\[ ||x||^2 \le c (\det L)^{2/m} \] 
where $c$ is a constant depending only on $m$. 
(Here, $||x|| = (\sum_{i=1}^m x_i^2)^{1/2}$ is the usual Euclidean norm.)
The best possible value of $c$ is called Hermite's constant and 
denoted $\gamma_m$. 
Blichfeldt \cite{Blichfeldt} has shown that 
\begin{equation}
\label{blichfeldt}
 \gamma_m^m \le \left(\frac{2}{\pi}\right)^m 
  \Gamma \left( 1 + \frac{m+2}{2}\right)^2. 
\end{equation}

Let $A$ be a central simple algebra over $\Q$ of dimension $n^2$.
For $n \in \{3,5\}$ the following argument gives a direct proof 
that if $A_{\Q_p} \isom \Mat_n(\Q_p)$ for all primes $p$ then
$A \isom \Mat_n(\Q)$. (This should be viewed as generalising the
geometry of numbers proof of the Hasse principle for conics over $\Q$.)
First let $\OO$ be a maximal order in $A$.
Since $n$ is odd we may trivialise $A$ over the reals,
and hence identify $\OO$ as a subring of $\Mat_n(\R)$.
We identify $\Mat_n(\R) = \R^{n^2}$ in the obvious way.
Then  \[  \Disc(\OO) = (\det B)^2 \Disc(\Mat_n(\Z)) \]
where $B$ is an $n^2$ by $n^2$ matrix whose rows are a basis for $\OO$.
Our local assumptions show by~(\ref{discA}) that $\Disc(\OO) = 1$.
Since $\Disc(\Mat_n(\Z)) = 1$ it follows that $\det \OO = |\det B| = 1$.
Hence by the geometry of numbers there is a non-zero matrix
$M \in \OO \subset \Mat_n(\R)$ 
with $||M||^2 \le \gamma_{n^2}$. Blichfeldt's bound~(\ref{blichfeldt}) gives 
\begin{align*}
\gamma_{9} & \le \frac{2}{\pi} \left(\frac{12!}{2^{12} 6!} 
\sqrt{\pi}\right)^{2/9} \approx 2.24065, \\ 
\gamma_{25} & \le \frac{2}{\pi} \left(\frac{28!}{2^{28} 14!} 
\sqrt{\pi}\right)^{2/25} \approx 4.29494. 
\end{align*}
Hence $||M||^2 < n$.  Applying the Gram Schmidt algorithm to the
columns of $M$, we can write $M = QR$ where $Q$ is orthogonal and $R$
is upper triangular, say with diagonal entries $r_1, \ldots, r_n$.
Then by the AM-GM inequality
\[ | \det M |^{2/n} = (\prod_{i=1}^n r_i^2 )^{1/n} \le \tfrac{1}{n}
\sum_{i=1}^n r_i^2 \le \tfrac{1}{n} ||R||^2 = \tfrac{1}{n} ||M||^2 < 1. \]
But $\det M$ is the reduced norm of an element of $\OO$, and
therefore an integer. Hence $\det M = 0$,
\ie $M$ is a zero-divisor. As we have seen, this implies
that $A \isom \Mat_n(\Q)$ (recall that $n$ is prime).

This proof suggests the following algorithm. 
Starting with a $\Q$-algebra $A$, known to be isomorphic to $\Mat_n(\Q)$,
we perform the following steps.
\begin{itemize}
\item Compute a maximal order $\OO \subset A$. 
(See for example  \cite{IvanyosRonyai}, \cite{Ronyai1990}, \cite{Fr},
or the {\Magma} implementation
by de Graaf.) % \check{Is there a paper to refer to for this?})
\item Trivialise $A$ over the reals. In practice (for $n$ odd) 
we split the algebra by a number field of odd degree, and then take 
a real embedding.
\item Use the real trivialisation to embed $\OO$ as a lattice in 
$\Mat_n(\R) \isom \R^{n^2}$. Then compute an LLL-reduced basis for $\OO$.
\item Search through small linear combinations of the basis elements
of $\OO$ until we find an element with reducible minimal polynomial.
If $n$ is prime we can then compute a trivialisation as 
described in \SSS\ref{zerodiv}.
\end{itemize}

In practice for $n \in \{3,5\}$ the first basis vector of $\OO$ is a
zero-divisor, and so no searching is required in the final stage.
(The bounds in the LLL-algorithm are unfortunately not quite strong
enough to prove this. In the case $n=3$ we were able to rectify this
by proving an analogue of Hunter's theorem \cite{H}, 
\cite[Theorem 6.4.2]{CohenGTM138}. We omit the details.)
For general $n$ the algorithm still finds
a basis for $A$ with respect to which the structure constants
are small integers, and is therefore worth applying before
attempting any other method (for example using norm equations).

We give some theoretical justification for the last remark.
Suppose $\OO \subset \Mat_n(\R)$ has basis $M_1, \ldots, M_{n^2}$.
As observed above, the $n^2$ by $n^2$ matrix $B$ whose $i$th row
contains the entries of $M_i$ has determinant $1$. So by Cramer's
rule the structure constants (defined by~(\ref{strM})) satisfy
\begin{equation}
\label{Cramer}
 | c_{ijk}| \le || M_i M_j || \prod_{s \not= k} || M_s ||. 
\end{equation}
The LLL algorithm bounds
$\prod_{i=1}^{n^2} || M_i ||$ by a constant depending only on $n$.
So either $|| M_i || < \sqrt{n}$ for some $i$, in which case
we have found a zero-divisor, or the $||M_i ||$ are bounded by a
constant depending only on $n$. In this latter case, by~(\ref{Cramer})
and the fact $|| M_i M_j || \le || M_i || \cdot ||M_j||$,
the structure constants are also bounded by a constant 
depending only on $n$. 
These constants turn out to be rather large -- but fortunately 
the method works much better in practice.

%=========================================================================

\section{Projecting to the rank $1$ locus}
\label{exteqns}

In this section we explain the ``fudge factor'' $1/y_T$ used in our
description (see \SSS\ref{threedesc}) of the Segre embedding method
in the case $n=3$. 

Let $E/K$ be an elliptic curve. 
We write $\tau_P : E \to E$ for translation
by $P \in E$. The \emph{theta group} of level $n$ for $E$ is
\[ \Theta_E = \{ (f,T) \in \Kbar(E)^\times \times E[n] : \divv(f) 
 = \tau_T^*(n (\origin)) - n (\origin) \} \]
with group law
\[ (f_1,T_1) * (f_2,T_2) = (\tau_{T_2}^*(f_1) f_2, T_1 + T_2). \]
It sits in an exact sequence 
\[ 0 \ra \Gm \stackrel{\alpha}{\ra} \Theta_E \stackrel{\beta}{\ra} E[n] 
\ra 0 \]
where the structure maps $\alpha$ and $\beta$ are given by $\alpha:
\lambda \mapsto (\lambda,\origin)$ and $\beta: (f,T) \mapsto T$. The commutator
is given by the Weil pairing, \ie $x y x^{-1} y^{-1} = \alpha
e_n( \beta x, \beta y)$ for all $x,y \in \Theta_E$. 

The construction of the obstruction algebra depends on an element 
$\eps \in (R \otimes_K R)^\times$. In \cite{paperI} it is shown that
we can take 
\[ \eps(T_1,T_2) = \frac{\phi(T_1)\phi(T_2)}{\phi(T_1+T_2)} \]
where $\phi : E[n] \to \Theta_E$ is any Galois equivariant set-theoretic
section for $\beta$.
This element has the property that
\begin{equation}
\label{eps-wp}
 \eps(T_1,T_2) \eps(T_2,T_1)^{-1} = e_n(T_1,T_2). 
\end{equation}
If we change $\phi$ by multiplying by an element $z \in R^\times$
(viewed as a map $E[n] \to \Kbar^\times$)
then $\eps$ is multiplied by $\partial z$. It is shown in 
\cite[Lemma 4.6]{paperI} that this does not change the obstruction 
algebra (up to isomorphism).

One choice of $\phi$ is to take 
\[ \phi(T) = (F_T ,-T)^{-1} = (\tau_T^*(1/F_T),T) \]
where the $F_T$ are the functions with divisor $n (T) - n (\origin)$ 
scaled as specified at the start of \SSS\ref{improve_eqns}.
Then
\begin{align*}
\phi(T_1) \phi(T_2) &= (\tau_{T_1}^* (1/F_{T_1}),T_1) * 
(\tau_{T_2}^*(1/F_{T_2}) , T_2) \\
&= (\tau_{T_1+T_2}^*(1/F_{T_1}) 
\tau_{T_2}^*(1/F_{T_2}),T_1+T_2) \\
& = \frac{\tau_{T_1+T_2}^*(F_{T_1+T_2})}{\tau_{T_1+T_2}^*(F_{T_1}) 
\tau_{T_2}^*(F_{T_2})} \phi(T_1+T_2) 
\end{align*}
This gives the formula~(\ref{defeps})
cited in \SSS\ref{improve_eqns}. We recall from
\cite[\S3]{paperI} that when $n = 2m-1$ is odd an alternative 
choice of $\eps$ (suggested by~(\ref{eps-wp})) is
\begin{equation}
\label{defeps-via-M}
 \eps(T_1,T_2) = e_n(T_1,T_2)^m. 
\end{equation}
This choice of $\eps$ corresponds to choosing $\phi$ so that 
$\iota(\phi(T)) = \phi(T)^{-1}$ and $\phi(T)^n = 1$ for all 
$T \in E[n]$, where $\iota: \Theta_E 
\to \Theta_E$ is the involution $(f,T) \mapsto (f \circ [-1],-T)$.
Indeed applying the involution $\iota$ to 
\[ \phi(T_1) \phi(T_2) = \eps(T_1,T_2) \phi(T_1+T_2) \]
gives
\[ \phi(T_1)^{-1} \phi(T_2)^{-1} = \eps(T_1,T_2) \phi(T_1+T_2)^{-1} \]
and so 
\[  e_n(T_1,T_2) = \phi(T_1) \phi(T_2) \phi(T_1)^{-1} \phi(T_2)^{-1} 
= \eps(T_1,T_2)^2 \]
as required.

The formulae~(\ref{defeps}) and~(\ref{defeps-via-M}) differ 
by $\partial u$ where $u \in R^\times$ satisfies
\begin{align*}
 \iota( u(T) F_T,-T) &= (u(T) F_T,-T)^{-1} \\
(u(T) F_T, - T)^n &= 1 
\end{align*}
equivalently
\begin{align}
\label{ueqn1} 
u(T)^2 F_T(P) F_T(T-P) & = 1 \\
\label{ueqn2}
u(T)^n \textstyle\prod_{i=0}^{n-1} F_T(P+iT) & =1
\end{align} 
where $P \in E$ is arbitrary, subject to avoiding the zeros and poles
of these functions.
Taking $P = mT$ in~(\ref{ueqn1}) gives
\[ u(T) = \pm 1/F_T(mT). \]
We check that the sign is independent of $T \in E[n] \setminus \{\origin\}$.
If Galois acts transitively on $E[n] \setminus \{\origin\}$ then this is 
already clear.
In general we use~(\ref{ueqn2}) and the following lemma.
\begin{Lemma} \label{signs} 
Let $n \ge 3$ be an odd integer
and $\origin \not = T \in E[n]$ a point of order $r$. 
Then the rational function
\[ g_i: P \mapsto F_T(P + i T) F_T(P + (1-i)T) \]
satisfies 
\[ g_i(\origin) = \left\{ \begin{array}{ll} u(T)^{-2} & \text{ if }
i \not\equiv 0,1 \pmod{r} \\ -u(T)^{-2} & \text{ if }
i \equiv 0,1 \pmod{r} \end{array} \right. \]
\end{Lemma}
\begin{Proof}
By~(\ref{ueqn1}) the rational function 
\[  h_i: P \mapsto F_T(-P + i T) F_T(P + (1-i)T) \]
is constant with value $u(T)^{-2}$. If $i \not\equiv 0,1 \pmod{r}$
then $g_i$ and $h_i$ take the same value
at $P = \origin$. If $i \equiv 0,1 \pmod{r}$ then an extra minus sign
arises since $F_T$ has a pole of odd order at $\origin$.  Indeed expanding
$F_T$ as a power series in $t = x/y$ about $\origin$ it is clear that the 
rational function $P \mapsto F_T(-P) /F_T(P)$ takes value 
$-1$ at $P = \origin$.
\end{Proof}

To compute the product in~(\ref{ueqn2}) we put $P=\origin$ in 
\[  \prod_{i=0}^{n-1} F_T(P+iT) = F_T(P + mT) 
\prod_{i=1}^{m-1} g_i(P)  \]
and use Lemma~\ref{signs}. Since $n/r$ is odd 
we find that $u(T) = -1/F_T(mT)$ for all $T \in E[n] \setminus \{\origin\}$.

In the case $n=3$ we recall that relative to the Weierstra{\ss} equation 
$y^2 = x^3 + ax +b$ we have 
\[ F_T(x,y) = (y-y_T) - \lambda_T(x-x_T) \]
where $\lambda_T$ is the slope of the tangent line at $T = (x_T,y_T)$.
Therefore $u(T) = - 1/ F_T(-T) = 2/y_T$. 

In \SSS\ref{improve_eqns} we took $\eps$ given
by~(\ref{defeps}).  In Sections~\ref{threedesc} 
and~\ref{compobs} we took $\eps$
given by~(\ref{defeps-via-M}).  This difference does not matter when
computing the obstruction algebra, but it does matter when we
subsequently compute equations using the Segre embedding method.  For
instance, if we used the wrong $\eps$ then it would not be true (after
projection to the trace zero subspace) that we get a curve in the rank
1 locus of $\PP(\Mat_n)$.  In view of \cite[Lemma 4.6]{paperI} the
situation is remedied by multiplying by the factor $u(T)$. (Here we
use the usual pointwise multiplication in $R$.)  In
\SSS\ref{threedesc} we used the factor $1/y_T$.  The constant $2$
(or indeed any scalar in $K^\times$) can be ignored since the scalar
matrices act trivially on projective space.

%=========================================================================

\section{Examples}
\label{sec:examples}

We refer to~\cite{6and12} for examples using our work on $3$-descent 
to find points of large height on elliptic curves over $\Q$.
Here we instead use $3$-descent
to construct explicit non-trivial elements of the Tate-Shafarevich group,
and to give examples to show that the kernel of the 
obstruction map is not a group. 
See the \Magma\ file stored with the arXiv version of this paper~\cite{paperIIIarXiv}
for further details of these examples.

\subsection{An element of \Sha[3]}
\label{example1}
Let $E/\Q$ be the elliptic curve
\[ y^2 + x y = x^3 + x^2 - 1154 x - 15345 \]
labelled 681b1 in \cite{CremonaTables}. This curve has generic
3-torsion in the sense that the map
$\rho_{E,3} : \GQ \to \GL_2(\Z/3\Z)$ 
is surjective. We work with the Weierstra{\ss} equation 
$y^2 = x^3 + a_4 x + a_6$ where $a_4=- 1496259$ and $a_6 =- 693495810$.
Relative to this Weierstra{\ss} equation a 3-torsion point is
given by $T=(x_T,y_T)$ where
\[ x_T = 12 u^6 - 36 u^2 + 2115, \quad
y_T = -2820 u^7 - 144 u^5 + 16920 u^3 - 662268 u, \]
and $u$ is a root of $f(X) = X^8 - 6 X^4 + 235 X^2 - 3$.
The slope of the tangent line at $T$ is
$\lambda_T = (3 x_T^2 + a_4)/(2 y_T) = -3 u^7 + 15 u^3 - 705 u.$

Using the algorithm in 
\cite{SchaeferStoll} (see \SSS\ref{sec:w1} for a summary)
we find that $\Sel^{(3)}(E/\Q) \isom (\Z/3\Z)^2$. One of the non-trivial
elements is represented by
\[ a = \tfrac{1}{18}(u^6 - u^4 - 9 u^3 - 5 u^2 - 27 u - 3). \]
In this example we find the corresponding plane cubic.

Let $L = \Q(u)$ and $M = L(v)$ where $v$ is a root of
\[ g(X) = \frac{f(X)}{X^2-u^2} = X^6 + u^2 X^4 + (u^4 - 6) X^2 + u^6 - 6 u^2 + 235. \]
We also put $L^+ = \Q(u^2)$ and $M^+ = L(v^2)$. Let $\sigma$
and $\tau$ be the automorphisms generating $\Gal(L/L^+)$ and
$\Gal(M/M^+)$. The polynomial $f(X)$ splits over $M$ with roots
$\pm u$, $\pm v$, $\pm u_{10}$ and $\pm u_{01}$. 
There are embeddings $\iota_{10}: L \to M$ and 
$\iota_{01} : L \to M$ given by 
$u \mapsto u_{10}$ and $u \mapsto u_{01}$. 
We choose $u_{10}$ and $u_{01}$ so that $\tau(u_{10}) = u_{01}$ 
and $\iota_{10}(T) + \iota_{01}(T) = T$.

Following the description in \SSS\ref{threedesc} we put
\[ R  = \Q \times L \quad \text{ and } \quad  
 R \otimes R = \Q \times L \times L \times L \times L \times M. \]
Then $\alpha = (1,a)$ and 
$\rho = (1,1,1,\sigma(a)/s,s,t/s)$ where $s= -u^2$ and
\begin{align*}
t = \tfrac{1}{486} & (u^7 - u^5 - 5 u^3 - 27 u^2 + 240 u - 27) v^4 \\
   & + \tfrac{1}{486} (-2 u^7 + 2 u^5 - 27 u^4 + 10 u^3 - 237 u + 27) v^2 \\
   & + \tfrac{1}{162} (-u^7 - 9 u^6 + u^5 + 86 u^3 + 54 u^2 - 240 u + 45).
\end{align*}
We put $\eps = (1,1,1,1,1,\zeta_3)$ where $\zeta_3 \in M$ is a
primitive cube root of unity. It is not worth recording our choice of 
$\zeta_3$ since a different choice only has the effect of 
reversing the order of multiplication in the obstruction algebra.

The basis for $L$ as a $\Q$-vector space suggested in \SSS\ref{compobs} is
\begin{align*}
u_1 & = 1, \\
u_2 &= \tfrac{1}{3} (-u^7 + 6 u^3 - 235 u), \\
u_3 & = \tfrac{1}{18} 
  (80 u^7 - 9 u^6 + u^5 - 481 u^3 + 54 u^2 + 18795 u - 2124), \\
& \hspace{4cm} \vdots \\
% u_4 & = 1/54 (9 u^7 - u^6 + u^4 - 54 u^3 + 5 u^2 + 2142 u - 240) \\
% u_5 & = 1/54 (9 u^7 - u^6 + u^4 - 54 u^3 + 5 u^2 + 2088 u - 240) \\
% u_6 & = 1/6 (-7 u^7 + u^6 + 42 u^3 - 7 u^2 - 1645 u + 237) \\
u_7 & = \tfrac{1}{54} (97 u^7 - 9 u^6 + 2 u^5 - 584 u^3
                   + 63 u^2 + 22785 u - 2133), \\
u_8 & = \tfrac{1}{54} (462 u^7 - 53 u^6 + 6 u^5 - u^4 
         - 2769 u^3 + 319 u^2 + 108549 u - 12423).
\end{align*}
Then $R$ has basis $r_1, \ldots, r_9$ where $r_1 = (1,0)$ and 
$r_{i+1} = (0,u_i)$. Let $A$ be the obstruction 
algebra $(R,+,*_{\eps \rho})$ with basis $\av_1, \ldots, \av_9$ 
corresponding to $r_1, \ldots ,r_9$. Then $\av_1$ is 
the identity, and left multiplication by $\av_2$ is given by
\begin{align*}
\av_2^2  & = -3 \av_3 - 3 \av_5 - 3 \av_6, \\ 
\av_2 \av_3 & = 2 \av_2 + 3 \av_3 + 3 \av_8, \\ 
& \hspace{2cm} \vdots \\
% \av_2 \av_4 & = 12 \av_1 - 3 \av_3 + 3 \av_5 + 3 \av_6 \\ 
% \av_2 \av_5 & = 27 \av_1 + 2 \av_2 - 3 \av_3 - 3 \av_7 \\ 
% \av_2 \av_6 & = 27 \av_1 - 4 \av_2 + 3 \av_7 - 3 \av_8 \\ 
% \av_2 \av_7 & = -2 \av_2 - 3 \av_3 + 3 \av_4 - 6 \av_6 - 3 \av_8 - 6 \av_9 \\ 
\av_2 \av_8 & = 7 \av_2 - 3 \av_3 + 9 \av_4 - 3 \av_8, \\ 
\av_2 \av_9 & = -3 \av_1 + 4 \av_2 + 3 \av_3 - 3 \av_5 
+ 6 \av_6 - 3 \av_7 + 3 \av_8. 
\end{align*}
We do not record the full table of structure constants, but note
that the above sample is typical in that most entries are single 
digit integers.
(Alternatively the full table may be recovered from the trivialisation 
given below.) The basis vectors $\av_i$ have minimal polynomials
\small
\begin{align*}
    X - 1, & \,\,
    X^3 + 162, \,\,
    X^3 - 12 X - 227, \,\,
    X^2, \,\,
    X^3 - 12 X - 470, \,\,
    X^3 - 12 X - 470, \\
    & X^3 - 147 X - 367, \,\,
    X^3 - 201 X + 1307, \,\,
    X^3 + 123 X + 254. 
\end{align*}
\normalsize
Notice that $\av_4$ is a zero-divisor, so in this example it is 
particularly easy to find a trivialisation.

The discriminant of $L$ is $3^{11} \cdot 227^4$ and the ideal
generated by $a$ is a cube. As predicted by 
Lemma~\ref{compdisc} the order with basis the $\av_i$ has discriminant 
$ |\det (\Trd (\av_i \av_j)) | = 3^{20} \cdot 227^{4}.$
A basis for a maximal order in $A$ is given by
\small
\begin{align*}
\bv_1 & = \tfrac{1}{3}(\av_1 + 56 \av_6 + 126 \av_7 + 101 \av_8 + 2438 \av_9), \\ 
\bv_2 & = \tfrac{1}{2043}(3 \av_2 + 38 \av_4 + 471 \av_5 + 95432 \av_6 + 50049 \av_7 + 75876 \av_8 + 1408079 \av_9),
\\ 
\bv_3 & = \tfrac{1}{2043}(\av_3 + 167 \av_4 + 543 \av_5 + 175106 \av_6 + 57658 \av_7 + 87258 \av_8 + 1872296 \av_9),
\\ 
\bv_4 & = \tfrac{1}{9}(\av_4 + 529 \av_6 + 2041 \av_9), \\ 
\bv_5 & = \tfrac{1}{3}(\av_5 + 159 \av_6 + 105 \av_7 + 160 \av_8 + 2802 \av_9), \\ 
\bv_6 & = \tfrac{1}{3}(\av_6 + \av_9), \\
\bv_7 & = \tfrac{1}{3}(\av_7 + 8 \av_9), \\
\bv_8 & = \tfrac{1}{3}(\av_8 + 8 \av_9), \\
\bv_9 & = \av_9
\end{align*}
\normalsize
with minimal polynomials
\begin{align*}
   & X^3 - X^2 + 67882988 X + 153570178243, \,\,
    X^3 + 46000395 X + 93752525874, \\
   & X^3 + 80434914 X + 198363227932, \,\,
    X^3 + 4444433 X + 1099577331, \\
   & X^3 + 84844655 X + 243745052250, \,\,
    X^3 - 3 X, \\
   & X^3 + 725 X + 3507, \,\,
    X^3 + 671 X + 9393, \,\,
    X^3 + 123 X + 254.
\end{align*}
In defining the $\bv_i$ we have not made use of the fact 
the $\av_i$ are already LLL-reduced with respect to a 
real trivialisation. The simplest way to correct for this is 
to run LLL on the rows of the change of basis matrix. So instead of 
the $\bv_i$ we consider the basis
\begin{align*}
\bv'_1 & = \tfrac{1}{2043}(-12 \av_2 - 63 \av_3 - 4 \av_4 - 73 \av_6 + 18 \av_7 + 8 \av_9), \\ 
\bv'_2 & = \tfrac{1}{2043}(-81 \av_2 - 28 \av_3 - 27 \av_4 + 18 \av_6 + 8 \av_7 + 54 \av_9), \\ 
& \hspace{4cm} \vdots \\
% \bv'_3 & = 1/681(-15 \av_3 - 8 \av_4 + 27 \av_5 + 27 \av_6 + 12 \av_8) \\ 
% \bv'_4 & = 1/2043(-27 \av_2 + 28 \av_3 + 21 \av_4 + 69 \av_5 + 75 \av_6 + 154 \av_7 - 45 \av_8 + 18 \av_9) \\ 
% \bv'_5 & = 1/2043(-15 \av_2 - 91 \av_3 + 49 \av_4 - 12 \av_5 + 67 \av_6 - 91 \av_7 - 81 \av_8 + 10 \av_9) \\ 
% \bv'_6 & = 1/2043(138 \av_2 - 55 \av_3 + 54 \av_4 - 27 \av_5 + 18 \av_6 + 20 \av_7 - 12 \av_8 + 135 \av_9) \\ 
% \bv'_7 & = 1/2043(-27 \av_2 + 83 \av_3 + 126 \av_4 - 30 \av_5 - 24 \av_6 - 73 \av_7 + 138 \av_8 + 18 \av_9) \\
\bv'_8 & = \tfrac{1}{2043}(-36 \av_2 + 37 \av_3 + 48 \av_4 + 138 \av_5 - 81 \av_6 - 173 \av_7 - 90 \av_8 + 24 \av_9), \\
\bv'_9 & = \tfrac{1}{2043}(-681 \av_1 - 27 \av_2 - 85 \av_3 - 9 \av_4 + 6 \av_6 - 73 \av_7 + 18 \av_9) 
\end{align*}
with minimal polynomials
  \[ X^2, \,\,  X^2, \,\,  X^2, \,\,  X^3 - X, \,\,  
     X^3 - X, \,\,  X^3, \,\,  X^3 - X, \,\,  X^3 - X, \,\,  X^2 + X. \]
Now every vector in our basis is a zero-divisor! (Recall that to 
find a trivialisation we only needed to find one zero-divisor.) 
Using the method in \SSS\ref{zerodiv} we find a trivialisation:
\small
\begin{align*}
\av_1 & \mapsto 
\begin{pmatrix}
1 & 0 & 0 \\
0 & 1 & 0 \\
0 & 0 & 1 
\end{pmatrix} &
\av_2 & \mapsto 
\begin{pmatrix}
6 & -6 & 3 \\
6 & -6 & 0 \\
0 & -9 & 0 
\end{pmatrix} &
\av_3 & \mapsto 
\begin{pmatrix}
-4 & -3 & -3 \\
0 & 2 & -3 \\
9 & -9 & 2 
\end{pmatrix} \\
\av_4 & \mapsto 
\begin{pmatrix}
0 & 0 & 0 \\
0 & 0 & 0 \\
12 & 15 & 0 
\end{pmatrix} &
\av_5 & \mapsto 
\begin{pmatrix}
2 & 3 & 3 \\
0 & -4 & 3 \\
18 & -18 & 2 
\end{pmatrix} &
\av_6 & \mapsto 
\begin{pmatrix}
2 & 9 & -6 \\
0 & 2 & -6 \\
-9 & 9 & -4 
\end{pmatrix} \\
\av_7 & \mapsto 
\begin{pmatrix}
-5 & 3 & 3 \\
0 & -11 & 3 \\
-9 & -9 & 16 
\end{pmatrix} &
\av_8 & \mapsto 
\begin{pmatrix}
1 & -12 & 3 \\
-12 & -8 & 3 \\
-9 & 9 & 7 
\end{pmatrix} &
\av_9 & \mapsto 
\begin{pmatrix}
7 & -9 & -3 \\
9 & -11 & 6 \\
15 & -15 & 4 
\end{pmatrix}. 
\end{align*}
\normalsize

We recall that $R = \Q \times L$ and $L$ has basis $u_1, \ldots, u_8$.
The space of quadrics vanishing on the projection of 
$C_\rho \subset \PP(R)$ to $\PP(L)$ has basis
\begin{align*}
 q_{1} &= z_4 z_6 - z_5 z_6 - z_5 z_7 + z_6 z_8 + z_7 z_8, \\
 q_{2} &= 2 z_1 z_6 - z_3 z_6 + z_4 z_5 + z_5^2 - z_5 z_8 - 2 z_6 z_7 + z_6 z_8, \\
 q_{3} &= -z_1 z_6 - z_3 z_6 - 2 z_4 z_5 + z_4 z_8 + z_5^2 - z_5 z_8 + z_6 z_8, \\
& \hspace{5cm} \vdots \\
 q_{18} &= z_1^2 - z_1 z_3 - z_1 z_4 + z_1 z_5 + 2 z_1 z_6 - z_1 z_7 - z_2 z_4
           + 2 z_2 z_5 - 2 z_2 z_6 - 2 z_3^2 \\
        & \quad{}
          + 2 z_3 z_4 - 2 z_3 z_5 - z_3 z_6 - z_3 z_7 + 3 z_3 z_8 + 2 z_4^2
          - z_4 z_5 + z_4 z_6 - 4 z_4 z_7 \\
        & \quad{}
          - z_4 z_8 + z_5 z_6 + 3 z_5 z_7 + z_5 z_8 + z_6^2 + 2 z_6 z_7
          + z_6 z_8 - z_7^2 - z_8^2.
\end{align*}

Multiplication by the factor $1/y_T \in L$ (relative to the basis
$u_1, \ldots, u_8$) followed by the above trivialisation,
prompts us to substitute
\small
\[ \begin{pmatrix} z_1 \\ z_2 \\ z_3 \\ z_4 \\ z_5 \\ z_6 \\ z_7 
  \\ z_8 \end{pmatrix} = 
\begin{pmatrix}
-70 & -54 & 144 & 12 & 86 & 423 & 6 & -48 & -16 \\
360 & 615 & 1128 & -444 & -510 & 663 & -372 & 207 & 150 \\
160 & -180 & -621 & 222 & -257 & 648 & 174 & 291 & 97 \\
-75 & -48 & 195 & -195 & 21 & -285 & -69 & -81 & 54 \\
3 & -24 & -267 & 24 & -192 & -21 & 87 & 81 & 189 \\
-69 & -81 & -108 & -9 & 57 & -135 & 117 & 36 & 12 \\
-105 & -81 & -270 & 18 & 129 & 27 & 9 & -72 & -24 \\
252 & -72 & -801 & 72 & -333 & 423 & 261 & 243 & 81 
\end{pmatrix}
\begin{pmatrix} z_{11} \\ z_{12} \\ z_{13} \\ z_{21} \\ z_{22} \\ 
z_{23} \\ z_{31} \\ z_{32} \\ z_{33} \end{pmatrix}. \]
\normalsize
Next we substitute $z_{ij} = x_i y_j$ to give 18 forms of \bidegree{} $(2,2)$.
Multiplying each of these by the $x_i$ gives 54 forms of \bidegree{} $(3,2)$.
We then solve by linear algebra for the ternary cubic $F_1$ 
(unique up to scalars) such that $y_1^2 F_1(x_1,x_2,x_3)$ belongs to 
the span of these forms:
\[ F_1(x,y,z) = 3 x^3 - 13 x^2 y + 4 x^2 z + 2 x y^2 + x y z 
                           - y^3 - 5 y^2 z - y z^2 + z^3. \]
This is the ternary cubic corresponding to $a$. 
Since $E(\Q)/3E(\Q) = 0$ it represents a non-trivial element 
of $\Sha(E/\Q)[3]$. In general we 
now minimise and reduce using the algorithms in \cite{minred234}.
However in this example we 
find that $F_1$ is already minimised and close to being reduced. 

Repeating for different $a$ we find that the other non-trivial 
elements of $\Sha(E/\Q)[3]$ are represented by
\small
\begin{align*}
F_2(x,y,z) & = x^3 + 6 x^2 y + 4 x^2 z + 4 x y^2 + 5 x y z + 2 x z^2 + y^3 - 3 y^2 z + 7 y z^2 
    + 6 z^3, \\
F_3(x,y,z) & = x^3 - 2 x^2 y - x^2 z - 7 x y z + 8 x z^2 + 4 y^3 - 5 y^2 z + 6 y z^2 + z^3, \\
F_4(x,y,z) & = x^3 - 2 x^2 z + 4 x y^2 + 3 x y z - 5 x z^2 - y^3 + 6 y^2 z + 2 y z^2 + 7 z^3.
\end{align*}
\normalsize
We recall that inverses in the $3$-Selmer group are represented by 
the same cubic with different covering maps. Thus our 3-descent
programs return a list of $(3^s - 1)/2$ ternary cubics where $s$ is 
the dimension of the Selmer group as an $\F_3$-vector space.

We have computed equations for all elements of $\Sha(E/\Q)[3]$ for all
elliptic curves $E/\Q$ of conductor $N_E < 230000$. The results can be
found on the website \cite{Sha3List}. In compiling this list we only
ran our programs on the elliptic curves with analytic order of $\Sha$
divisible by $3$, and did not compute the class groups
rigorously. Thus the completeness of our list remains conditional on
the Birch--Swinnerton-Dyer conjecture.\footnote{However, for curves of
  rank~$0$ and~$1$ and conductor ${}<5000$ the verification of the full
  BSD conjecture has recently been completed
  by B.~Creutz and R.L.~Miller~\cite{CreutzMiller}.}
It is however unconditional that every
cubic in our list is a counterexample to the Hasse Principle.

\subsection{Adding ternary cubics}
We give two examples to show that the kernel of the 
obstruction map for $3$-coverings is not a group. 
These generalise the example for $2$-coverings given 
in \cite[\SSS5]{Cremona2Desc}.

Let $E/\Q$ be the elliptic curve 
\[ y^2 + x y + y = x^3 - x^2 + 40 x + 155 \]
labelled 126a3 in \cite{CremonaTables}.
The ternary cubics
\begin{align*}
F_1(x,y,z) &=   x^3 + x y^2 - x y z + x z^2 + y^3 + 3 y^2 z - 6 y z^2 + z^3 \\
F_2(x,y,z) &=   2 x^2 y + 2 x^2 z + 3 x y^2 - x y z + 2 x z^2 + 2 y z^2 + 2 z^3
\end{align*}
represent $3$-coverings of $E$ with $\Delta(F_1) = \Delta(F_2) 
= \Delta_E = -2^6 \cdot 3^6 \cdot 7^3$. Whereas the second of these
has the obvious rational point $(1:0:0)$, the first is not
locally soluble at the primes $2$ and $7$.

Let $K = \Q(\zeta_3)$ where $\zeta_3$ is a primitive cube root of unity.
Then $E(K)[3] \isom (\Z/3\Z)^2$ generated by
$P = (1 , 13)$ and $Q = (-6, 13 + 21 \zeta_3)$.
By the formulae in \cite{testeqtc} these points 
act on the first cubic via \[ M^{(1)}_P =
\begin{pmatrix}
1 & 0 & 0 \\
0 & -1 & 1 \\
0 & -1 & 0 
\end{pmatrix} \quad
M^{(1)}_Q =
\begin{pmatrix}
0 & 1-\zeta_3 & 1+ 2 \zeta_3  \\
1 & 1 & -1 -\zeta_3  \\
1 + \zeta_3 & -\zeta_3 & -1 
\end{pmatrix} 
\]
and on the second cubic via
\[
M^{(2)}_P =
\begin{pmatrix}
0 & -3 & -1 \\
2 & 2 & -2 \\
-2 & 0 & -2 
\end{pmatrix} \quad
M^{(2)}_Q =
\begin{pmatrix}
1 & 4 + 2 \zeta_3 & -4 \zeta_3 \\
4 + 2 \zeta_3  & 1-\zeta_3  & 2+ 4 \zeta_3  \\
-2 \zeta_3 & 1 + 2 \zeta_3  & -2 + \zeta_3  
\end{pmatrix}.
\]
Taking determinants shows that the 3-coverings are represented
as elements of 
$H^1(K,E[3]) \isom K^\times/(K^\times)^3 \times K^\times/(K^\times)^3$ by
\[\alpha_1 = (1,\zeta_3^2) \quad \text{ and } \quad 
\alpha_2 = \left(28,\frac{1-2 \zeta_3^2}{1-2 \zeta_3}\right). \]

We recall from \cite{PeriodIndex} that in this split torsion case, 
the obstruction map 
\[\Ob_3 : H^1(K,E[3])  \to \Br(K)[3] \subset \mathop{\text{\Large$\oplus$}}\limits_\pp \Br(K_\pp)[3] \]
is given by the local $3$-Hilbert norm residue symbols, 
subject to identifying
$\Br(K_\pp)[3] \isom \frac{1}{3}\Z/\Z \isom \mu_3$.
It is routine to check (see for example 
the exercises in \cite{CasselsFrohlich})
that $(28,(1-2 \zeta_3^2)/(1-2 \zeta_3))_\pp =1$
for all primes $\pp$ of $K$, but
\[  (28,\zeta_3)_\pp = \left\{ \begin{array}{ll} 
     \zeta_3 & \text{ if $\pp \mid 2$ or $\pp \mid 7$} \\
    1 & \text{ otherwise. } \end{array} \right. \]
Thus $\Ob_3(\alpha_1) = \Ob_3(\alpha_2) = 0$ yet 
$\Ob_3(\alpha_1 \alpha_2)\not=0$. This shows that the sum of our 
two $3$-coverings cannot be represented as a ternary cubic over $\Q$
(or even $K$). In particular the kernel of the obstruction map is not 
a group. 

We give a second example to show that this behaviour is not peculiar
to the split torsion case. Let $E/\Q$ be the elliptic curve
\[ y^2 + x y + y = x^3 - 43 x - 490 \]
labelled 1722f1 in \cite{CremonaTables}. The Galois action on the
$3$-torsion of $E$ is generic.
A non-trivial $3$-torsion point is 
\begin{align*}
  T = \big(&\tfrac{1}{192} (u^6 + 9 u^4 + 315 u^2 + 1979), \\
           &\tfrac{1}{44928} (-643 u^7 - 117 u^6 - 1755 u^5 - 1053 u^4 - 166257 u^3 \\
           &\hphantom{\tfrac{1}{44928} (}{}- 36855 u^2 - 888689 u - 254007) \big)
\end{align*}
defined over $L= \Q(u)$ where 
$u$ is a root of $X^8 + 234 X^4 + 1256 X^2 - 4563$.

The ternary cubics 
\begin{align*}
F_1(x,y,z) & = x^3 - 2 x^2 z + 2 x y^2 + x y z + 3 x z^2 + y^3 + 3 y^2 z - y z^2 + 2 z^3 \\
F_2(x,y,z) & = 3 x^2 y + x^2 z - x y^2 + 3 x y z - 2 x z^2 + y^3 + 6 y z^2 + z^3 
\end{align*}
represent $3$-coverings of $E$ with 
\[ \Delta(F_1) = \Delta(F_2) = \Delta_E  = -2^8 \cdot 3^3 \cdot 7^3 \cdot 41. \]
Whereas the second of these has the obvious rational 
point $(1:0:0)$, the first is not locally soluble at the primes $3$ and $7$.
The corresponding elements of $L^\times/(L^\times)^3$, computed 
using the formula in \cite{testeqtc}, are
\begin{align*}
a_1 &= \tfrac{1}{13312} (-11 u^7 - 65 u^6 - 39 u^5 - 117 u^4
           - 2561 u^3 \\ & \hspace{11em} - 16419 u^2 - 20173 u - 126503), \\
a_2 &= \tfrac{1}{6656} (-253 u^7 + 364 u^6 - 793 u^5 + 1092 u^4
            - 58695 u^3 \\ & \hspace{11em} + 81172 u^2 - 457635 u + 616252).
\end{align*}
We now attempt to compute a cubic corresponding 
to $a = a_1 a_2$. The basis for $L$ suggested 
in \SSS\ref{compobs} is
\small
\begin{align*}
u_1 &= \tfrac{1}{39936} (15 u^7 + 13 u^6 - 65 u^5 + 13 u^4 + 4069 u^3
        + 3055 u^2 - 2675 u - 10569), \\
u_2 &= \tfrac{1}{19968} (-7 u^7 + 26 u^6 + 65 u^5 - 130 u^4 - 1885 u^3
        + 6422 u^2 + 5859 u - 8814), \\
& \hspace{5cm} \vdots \\
% u_3 &= 1/19968 (-4 u^7 - 13 u^6 + 26 u^5 + 65 u^4 - 1144 u^3 - 5083 u^2 - 318 u + 
%        8775), \\
% u_4 &= 1/39936 (-9 u^7 - 65 u^6 + 91 u^5 + 403 u^4 - 2795 u^3 - 17459 u^2 + 15385 u
%        + 52065), \\
% u_5 &= 1/19968 (4 u^7 + 13 u^6 - 26 u^5 - 65 u^4 + 1144 u^3 + 2587 u^2 + 318 u + 
%        3705), \\
% u_6 &= 1/39936 (3 u^7 + 13 u^6 - 65 u^5 - 143 u^4 + 169 u^3 + 2119 u^2 - 9323 u - 
%        1989), \\
u_7 &= \tfrac{1}{19968} (7 u^7 - 13 u^6 - 13 u^5 + 143 u^4 + 1781 u^3
         - 4615 u^2 - 3311 u + 14469), \\
u_8 &= \tfrac{1}{13312} (3 u^7 + 13 u^6 - 13 u^5 - 195 u^4 + 481 u^3
         + 3471 u^2 + 1129 u - 7449).
\end{align*}
\normalsize

Proceeding exactly as in \SSS\ref{example1} we compute structure constants 
for the obstruction algebra $A$. Again $\av_1$ is the identity and
a portion of the multiplication table (describing left multiplication 
by $\av_2$) is as follows.
\begin{align*}
\av_2^2 & = 56 \av_1 - 2 \av_2 - 2 \av_3 + 2 \av_5 + 3 \av_6 - 2 \av_7 + 2 \av_8 + \av_9, \\ 
\av_2 \av_3 & = \av_1 + 6 \av_2 + 4 \av_3 + 7 \av_4 + 2 \av_5 - 8 \av_6 + 7 \av_7 - 4 \av_8 + \av_9, \\ 
& \hspace{4cm} \vdots \\
% \av_2 \av_4 & = -36 \av_1 + 4 \av_2 + 6 \av_3 + 2 \av_4 - 3 \av_6 - 2 \av_7 - 2 \av_8 + \av_9 \\ 
% \av_2 \av_5 & = -18 \av_1 - 3 \av_3 - 4 \av_4 - \av_5 + 7 \av_6 + 2 \av_8 \\ 
% \av_2 \av_6 & = 19 \av_1 + 12 \av_2 + 8 \av_3 + 7 \av_4 + 9 \av_5 - 10 \av_6 + 4 \av_7 - 9 \av_8 - 2 \av_9 \\ 
% \av_2 \av_7 & = -9 \av_1 + 6 \av_3 + 3 \av_5 - 3 \av_6 - 4 \av_7 - 3 \av_8 - \av_9 \\ 
\av_2 \av_8 & = 28 \av_1 - 6 \av_2 - 4 \av_3 - \av_4 - 5 \av_5 + \av_6 - 5 \av_7 + 7 \av_8 + 2 \av_9, \\ 
\av_2 \av_9 & = -39 \av_1 + 12 \av_2 + 6 \av_3 + 3 \av_4 - \av_5 - 2 \av_6 - \av_8 + 4 \av_9. 
\end{align*}

We have $(a) = \mathfrak{p} \mathfrak{q}^2 \mathfrak{c}^3 $ where
$\mathfrak{p}$ and $\mathfrak{q}$ are distinct primes of norm $7^3$.
As predicted by Lemma~\ref{compdisc} the order with basis the 
$\av_i$ has discriminant
$2^4 \cdot 3^{16} \cdot 7^6 \cdot 41^4 = 3^9 \cdot 7^6 \cdot \Disc(L)$.

We computed a maximal order and found it has discriminant $3^6 \cdot 7^6$.
It follows by~(\ref{discA}) that $A$ does not split.
Alternatively we may check this by reducing to a norm equation (see for
example \cite{GHPS}). To this end we put
\begin{align*}
 u &= \tfrac{1}{246}(-82 \av_1 - 21 \av_2 - 32 \av_3 - 4 \av_4 - 37 \av_5
                     + 26 \av_6 + 23 \av_7 + 25 \av_8 + 13 \av_9), \\
 v &= \tfrac{1}{246}(256 \av_2 + 225 \av_3 - 80 \av_4 + 166 \av_5 - 224 \av_6
                     - 335 \av_7 - 81 \av_8 + 104 \av_9).
\end{align*}
The minimal polynomial of $u$ is $X^3 + X^2 - 4 X + 1$ with 
discriminant $13^2$. Moreover $v u v^{-1} = u^2 + u - 3$ and $v^3 = b$ 
where $b = 3^2 \cdot 5 \cdot 7^2$.
Thus $A$ is the cyclic algebra $(F/\Q,\sigma,b)$ where
$F = \Q(u)$ and $\sigma: u \mapsto u^2 + u - 3$. In particular 
$A$ splits if and only if there exists $\theta \in F$ with 
$N_{F/\Q}(\theta) = b$. Since $3$ and $7$ are inert 
in $F$ the norm equation is not locally soluble at these primes.

In conclusion the $3$-coverings defined by $F_1$ and $F_2$ 
have trivial obstruction, but their sum does not. 

%=========================================================================

\end{document}